\documentclass[10pt]{article} 
\usepackage{amssymb}
\usepackage{amsmath}
\usepackage{mathtools}
\usepackage{epic}
\usepackage{accents}
\usepackage{enumitem, dsfont}
\usepackage{comment, csquotes}
\usepackage{physics}
\usepackage{mathrsfs}
\usepackage{amsthm}
\newtheorem{maintheorem}{Theorem}

\newtheorem{theorem}{Theorem}[section]
\newtheorem{lemma}[theorem]{Lemma}
\newtheorem{conjecture}[theorem]{Conjecture}
\newtheorem{question}[theorem]{Question}
\newtheorem{proposition}[theorem]{Proposition}

\newtheorem{corollary}[theorem]{Corollary}

\newtheorem{definition}[theorem]{Definition}

\newtheorem{remark}[theorem]{Remark}

\DeclareSymbolFont{bbold}{U}{bbold}{m}{n}
\DeclareSymbolFontAlphabet{\mathbbold}{bbold}

\def \C{\mathbb{C}}
\def \Z{\mathbb{Z}}
\def \Q{\mathbb{Q}}
\def \N{\mathbb{N}}
\def \R{\mathbb{R}}
\def \K{\mathbb{K}}

\def \ge{\geqslant}
\def \le{\leqslant}

\newcommand{\cC}{\mathcal{C}}

\newcommand{\cF}{\mathcal{F}}
\newcommand{\cG}{\mathcal{G}}

\newcommand{\an}{\mathrm{an}}

\newcommand{\cCexp}{\cC^{\mathrm{exp}}}

\renewcommand{\Re}{\operatorname{Re}}

\title{Approximating parametric suprema for constructible and power-constructible functions}

\author{Tijs Buggenhout, Mathias Stout, Lisa Vandebrouck}
\date{}

\begin{document}
	
	\maketitle
	
	\begin{abstract}
		We prove that one may approximate parametric suprema of constructible and power-constructible functions using functions within the same class.
		This resolves a conjecture by Adiceam and Cluckers, which was posited in a paper by Adiceam and Marmon (2023) after studying a question posed by Sarnak (1997).
		We apply our result to prove that a certain subclass of $\mathcal{C}^{\exp}$-class distributions is tempered, partially answering a question by Aizenbud, Cluckers, Raibaut and Servi (2023), and to make uniform a bound concerning pushforward measures, as mentioned in work by Glazer, Hendel and Sodin (2024).
%
	\end{abstract}
	
	
	
	\textbf{Abbreviated title: } Approximating parametric suprema

	

\section{Introduction}

	In real analysis, functions can classically be very pathological despite having several desirable properties.
	Nevertheless, in applications many of the studied functions behave rather tamely.
	Through the use of more specialized tools, like o-minimal geometry and real analytic geometry, stronger results can be acquired.
	Such geometries are generally called \emph{tame}.
	Currently, tame geometry is successfully being applied in other areas of mathematics and even physics; see for example Frésan's exposition on the use of o-minimality in Hodge theory \cite{F2020} or Grimm's recent work in physics, like \cite{G2026}.
	
	Our focus will be on the theory of \textit{constructible functions}, which arise as integrals of globally subanalytic functions.
	This study was started by Lion and Rolin in 1998 \cite{LR1998} and further expanded upon by, among others, Cluckers, Comte, Lion, Miller, Rolin (e.g. in \cite{CLR2000}, \cite{CM2011}, \cite{CM2013}).	
	The development of the theory of constructible functions happened in parallel with a non-archimedean analogue which was started by Denef in 1984 \cite{D1984} and expanded upon by, among others, Cluckers, Gordon and Halupczok (e.g. in \cite{CGH2014}, \cite{CGH2018}, \cite{CH2018}).	
	In both the real and the $p$-adic case, the main distinguishing property of the algebra of constructible functions is its closedness under parametric integration, as shown in \cite{CM2011}, \cite{CM2012}, \cite{CL2008}.
	
  In fact, we will be looking at an algebra of functions which are a bit more general, called \emph{power-constructible} functions, as introduced in \cite{CCRS24}.
  For a given subfield $\K \subset  \C$, we roughly consider the algebra of constructible functions enriched by the maps $x \mapsto x^r$ for $r \in \K$.
  In the same paper, it is proven that this algebra is stable under parametric integration as well.
  
In this paper, we will show that one can approximate parametric suprema of $\K$-constructible functions using other $\K$-constructible functions, in a suitable sense.
A $p$-adic analogue of this result has been shown for the algebra of constructible functions in \cite{CGH2018}. Studying approximate suprema is inherently difficult, even for Presburger-definable functions (\cite[Lem.\,2.2.3,2.2.4]{CGH2018}).
Note that we can not get around this by using definability in $\mathbb{R}_{\an, \exp}$, as this would leave the studied class. 
Moreover, if $\K \not \subset \R$, then functions of $C^\K$-class do not even live in any o-minimal structure.
\begin{maintheorem}\label{th:main_intro}
		Let $\K \subset \C$ be a subfield and $\mathcal{C}^{\K}(Y)$ be the algebra of $\K$-power constructible functions $Y \to \R$, where $Y \subset \R^{m+n}$ is subanalytic.
		Let $X=\pi_m(Y)$ be the projection of $Y$ to the first $m$ coordinates.

		Let $f \in \cC^\K(Y)$ and suppose that for any $x$, the map $f_x: Y_x \to \R: y \mapsto f(x,y)$ is bounded (where $Y_x=\{y \in \R^n \mid (x,y) \in Y\}$).
		Then there exist finitely many maps $g_1, g_2, ..., g_k \in \mathcal{C}^{\K}(X)$ and a constant $C>0$ such that
		
		\[\frac1C \max\{\abs{g_1(x)}, ..., \abs{g_k(x)}\} \le \sup_{y \in Y_x} \abs{f(x,y)} \le C \max\{\abs{g_1(x)}, ..., \abs{g_k(x)}\}\]
		
		for all $x \in X$.
		
		Moreover, if $\K \subset \R$, then there exist subanalytic maps $a_1, ..., a_k$ such that we may take $g_i=f \circ a_i$.
		
	\end{maintheorem}
		
	
	The main motivation for our research question comes from a conjecture by Adiceam and Cluckers, see~\cite[Conj.\,6.9]{AM2023} (to appear in~\cite{AM2025}). 
	Corollary \ref{conj cluckers adiceam} below resolves it in the affirmative.
%
	\begin{corollary} \label{conj cluckers adiceam}
		Let $F: X \subset (0, +\infty) \times \R^n \to \R_{\ge 0}: (\varepsilon, x) \mapsto F(\varepsilon, x)$ be a non-negative constructible function such that $F_\varepsilon$ is bounded above for all sufficiently small $\varepsilon$, and $F_\epsilon$ does not vanish identically for small $\epsilon$.
		Then there exist a rational number $r$, an integer $l$ and a constant $C>0$ such that
		
		$$\frac1C \varepsilon^r \abs{\log \varepsilon}^l \le \sup_{x\in X_\varepsilon} F(x, \varepsilon) \le C \varepsilon^r \abs{\log \varepsilon}^l$$
		
		for all sufficiently small $\varepsilon$.
	\end{corollary}
We briefly outline the context of this conjecture. 		
Denote $F = (F_1, \cdots, F_p)$, where $\{F_1(x), \cdots, F_p(x)\}$ is a set of real homogeneous forms of the same degree on $\R^n$, with $p \geq 1$ and $n \geq 2$, and denote $\operatorname{Vol}_n$ for the Lebesgue measure on $\R^n$.
The following theorem is proven in ~\cite[App.\,1]{S97}. 

\begin{theorem} [Sarnak, \cite{S97}] \label{thm sarnak intro}
	Assume that:
	\begin{itemize}
		\item The set $\{x \in K: F(x) = 0\}$ does not lie in an $(n-p)$-dimensional linear subspace.
		\item  For all $x \in K$, the gradient vectors $\nabla F_1(x), \cdots, \nabla F_p(x)$ are linearly independent.
		\item $K$ is \enquote{nice}.
	\end{itemize}
	Then the following statement holds: for all $\alpha > 0$, there exist $C, \delta > 0$, such that \begin{equation} \label{eq Sarnak}
		\#  \bigl\{\{x \in T \cdot K: \norm{F(x)} \leq T^{d-\alpha}\} \cap \Z^n\bigr\}  \leq C \left( T^{n-p\alpha} + T^{n-p-\delta} \right) .\end{equation}
\end{theorem}
Sarnak does not explain what he means with \enquote{nice}, however \cite{AM2023} states that the following list of conditions on $K$ is sufficient to make the proof work.
	\begin{enumerate} [label=(P\arabic*)]
	\item $K$ is compact, equal to the closure of its interior, and it intersects the variety $\{x \in \R^n \mid F(x)=0\}$ non-trivially (i.e.\, the intersection is non-empty).
	\item $K$ is convex and has piecewise smooth boundary.
\end{enumerate}
Note that when $\alpha>1$, the right-hand-side of Equation \eqref{eq Sarnak} provides a non-trivial bound for the number of solutions $m \in T \cdot K$ to $F(m)=0$.
Sarnak then conjectures the following result at the end of his proof:
\begin{conjecture}[Sarnak, \cite{S97}] \label{conj: sarnak}
	Assume that:
	\begin{itemize}
		\item The set $\{x \in K \mid F(x) = 0\}$ does not lie in an $(n-p)$-dimensional linear subspace and is non-empty.
		\item $K$ satisfies (P1).
	\end{itemize}
	Then the following statement holds: then for all $\alpha > 0$, there exist $C,\delta > 0$, such that for all $T > 0$
\begin{equation}\label{eq:conjecture_Sarnak}
	\begin{aligned}
		\#\Big( \{x \in T \cdot K : \norm{F(x)} \le T^{d-\alpha}\}
		\cap \mathbb{Z}^n \Big)
		\leq& C \Big( \operatorname{Vol}_n \Big(
		\{x \in T \cdot K : \norm{F(x)} \le T^{d-\alpha}\}
		\Big) \\
		&\quad + T^{\,n-p-\delta} \Big).
	\end{aligned}
\end{equation}
\end{conjecture}
In \cite{AM2023}, the authors provide an example illustrating that additional assumptions are necessary for the conjecture to hold, and they explain that problems arise if the set $\{x \in \R^n \mid F(x)=0\}$ is in some sense \enquote{too flat}. 
This is possible even if it is not lying in a $(n-p)$-dimensional linear subspace. 
This motivates the following research question examined in \cite{AM2023}. 
\begin{question}\label{question AM} 
	Assume that $K$ is semi-algebraic and satisfies the same conditions as in Conjecture \ref{conj: sarnak}. 
	Which additional assumptions related to the \enquote{level of flatness} of the set $\{x \in \R^n \mid F(x)=0\}$ should be imposed such that Equation \eqref{eq:conjecture_Sarnak} holds?
\end{question}
This measure of flatness is rigorously introduced in \cite{AM2023}, where a theorem is shown proving Sarnak's conjecture under additional assumptions (Theorem 6.1 in \cite{AM2023}).
Corollary \ref{conj cluckers adiceam} makes the definition of the measure of flatness effective, in the sense that a $\liminf$ becomes an actual limit, with explicit error term for the convergence ~\cite[p\, 116]{AM2023}.
We will discuss this application further in Section \ref{sec:corollaries}.

	
	Other than the positive resolution of the Adiceam-Cluckers conjecture, we also use Theorem \ref{th:main_intro} to prove that the algebra of constructible functions is polynomially bounded, in a precise and subtle sense which is different from the more common notion of polynomially
	bounded o-minimal structures, as we explain in section \ref{subsec: poly bounded}, and to provide a partial answer to the question in \cite{ACRS24} asking whether all $\mathcal{C}^{\exp}$-distributions are tempered. Moreover, we apply it to make a bound for certain integrals involving pushforward measures, given in \cite{GHS24}, uniform, as suggested in \cite{GHS24}.
	
	The natural idea to prove Theorem \ref{th:main_intro} is to adapt Cluckers, Gordon and Halupczok's strategy in the $p$-adic case \cite{CGH2018}.
	They proceed by reducing the problem to bounding $\max_{0 \le w \le t} |h(w)|$ or $\sup_{w \in \N} |h(w)|$, for $h: \N \to \C$ of the form
	
	\begin{equation}\label{eq:CGH2018} 
		h(w)=\sum_i c_i w^{a_i} q^{b_i w},
		\end{equation}
	
	independently of the coefficients $c_i$.
	Their reduction uses model-theoretic techniques to get from a function over the valued field to a function over the residue field and value group.
	However, a direct adaptation runs into difficulties. Indeed, in our case, there are three complicating factors.
	
	First of all, we cannot expect to reduce to functions of the form \eqref{eq:CGH2018}, even after the substitution $w=\log y$.
	Indeed, as a subanalytic function can locally look like any analytic function, the coefficients $c_i$ will usually have to depend on $y$ as well.
	Hence, instead of reducing to functions of the form \eqref{eq:CGH2018}, we reduce to functions of the form
	
	\begin{equation}\label{eq:intro_reduction}
		h(y)=\sum_i c_i f_i(y) y^{a_i} (\log y)^{b_i},	\end{equation}
		
	where the $f_i(y)$ can depend on $y$ but are close to $1$ in a suitable sense.
	We want to bound the supremum of such functions independently of the coefficients $c_i$ and the functions $f_i$.
	This happens mostly in Section \ref{sec:unbalanced}.
	The proof is based on the proof of Lemmas 2.2.3 and 2.2.4 in \cite{CGH2018}, but we need more precise estimates to make the argument work.
	Furthermore, we need the Riemann-Lebesgue lemma to deal with oscillatory terms stemming from complex powers.
	
	A second complication is that, in order to reduce to functions of the form \eqref{eq:intro_reduction}, we do not have access to the model theory of valued fields.
	Hence, we need a different strategy altogether.
	The main ingredient here is Cluckers and Miller's rectilinear preparation theorem for subanalytic functions (\cite{CM2013}, Theorem 1.5).
	We use it in Section \ref{sec:preparation results} to get several preparation results for power-constructible functions, eventually reducing to functions of the form \eqref{eq:intro_reduction}.
	
	A third and final challenge consists in bounding the parametric supremum of power-constructible maps $f:\R^n \times C \to \C$, where $C \subset \R^m$ is compact.
	In the $p$-adic case, this corresponds to bounding the parametric supremum of $h:X \times \{0,1, ..., n\} \to \C$, which equals $\max_{i \in \{0,1, ..., n\}} |h(x, i)|$ and is therefore automatically of the desired form.
	However, over $\R$, we still need to work with a continuum.
	We will call this situation \emph{balanced} and the other situation \emph{unbalanced}.
	We use a topological argument using the compactness as a workaround, which happens in Section \ref{sec:balanced}.
	
	Finally, one may hope that instead of using the maximum of several power-constructible functions, Theorem \ref{th:main_intro} would work without the maximum or without the absolute values.
	We give several examples illustrating why this is not possible in general in Section \ref{sec:counterexamples}.
	
	\medskip\noindent
	\textbf{Structure of the paper.} We start by giving the necessary definitions and preliminary results in Section \ref{sec:prelim}.
	Then we show the necessary preparation results in Section \ref{sec:preparation results}.
	The proof of Theorem \ref{th:main_intro} falls into two parts, for two kinds of cells, which are discussed in Section \ref{sec:unbalanced} and Section \ref{sec:balanced}.
	These two sections complete the proof of Theorem \ref{th:main_intro}, so that we can discuss its applications in Section \ref{sec:corollaries}.
	Finally, we give several counterexamples in Section \ref{sec:counterexamples}, showing that Theorem \ref{th:main_intro} is the best possible in a suitable sense.
	
	\medskip\noindent
	\textbf{Acknowledgements.}
	We would like to thank Neer Bhardwaj, Faustin Adiceam and Raf Cluckers for interesting discussions on the topics of the paper.
	We would like to thank Yotam Hendel for pointing out and discussing the Corollary~\ref{corol: pushforward measures}. T.B. was supported by FWO
	Flanders (Belgium) with grant number 1131925N.
	M.S. was supported by McMaster University and the Fields Institute. L.V. was supported by KU Leuven IF grant C16/23/010.
	
%
%
%

	\section{Preliminaries} \label{sec:prelim}
	
	\subsection{Notational conventions}
	
	Denote by $\N=\{0, 1, 2, ...\}$ the set of natural numbers.
	For a set $X \subset \R^n$, $\overline{X}$ is its closure in the Euclidean topology.
	
	Given a tuple $y \in \R^n$, we will denote by $y_{\Box l}$ the tuple of $y$-variables $(y_i)_{i \Box l}$, where $\Box \in \{<, \le, >, \ge\}$.
	We will say $y \ge 0$ if $y_i \ge 0$ for all $i$ and we will say $y>0$ if $y \ge 0$ and $y \neq 0$.
	We define $y \le 0$ and $y<0$ similarly.
	If $\alpha \in \C^n$ and the power $y_i^{\alpha_i}$ is defined for each $i$, then we let $y^\alpha=y_1^{\alpha_1}\ldots y_n^{\alpha_n}$.
	For $y \in (\R_{>0})^n$, we denote $\log y=(\log(y_1), ..., \log(y_n))$.
	
	We denote
	
	\[\pi_m: \R^{m+n} \to \R^m: (x_1, ..., x_{m+n}) \mapsto (x_1, ..., x_m)\]
	
	for all $m, n \in \N$.
	
	We say that $f: X \subset \R^n \to \R$ is analytic if there exists an open neighbourhood $U$ of $X$ and a real analytic function $g: U \to \R$ such that $g|_X=f$.
	
	Given $Y \subset \R^{m+n}$ and $x \in \R^m$, let
	
	\[Y_x=\{y \in \R^n \mid (x,y) \in Y\}.\]
	
	Given $f: Y \subset \R^{m+n} \to \R$ and $x \in \R^m$, we let $f_x: Y_x \to \R: y \mapsto f(x,y)$.
	We say that $f$ is \emph{fiberwise bounded} if $f_x$ is bounded for all $x \in \R^m$ (note that this is a slight abuse of notation as the subdivision of $\R^{m+n}$ into $\R^m$ and $\R^n$ is not unique; however, the one meant will always be clear from context).
	If $f$ is fiberwise bounded, we denote
	
	$$S_f: \pi_m(Y) \to \R: x \mapsto \sup_{y \in Y_x} \abs{f_x(y)}\\
	$$
	
	and we say that $S_f$ is a \emph{parametric supremum}.
	Finally, given $a, b: X \to \R$, we say that $a \sim b$ if there exists a constant $C>0$ such that
	
	\[\frac1C a \le b \le Ca\]
	
	on $X$.
	We call $C$ `the' \emph{implicit constant}.
	
	\subsection{O-minimality}
	
	Throughout this paper, we will strongly rely on the machinery from o-minimal geometry; see, for instance, \cite{vanDenDries1998TameTopology} by van den Dries.
	
	\begin{definition}
		A structure $\mathcal{U}$ on $\R$ is called o-minimal if every definable subset of $\R$ is a finite union of open intervals and points.
	\end{definition}
	
	The main property of o-minimal structures we need is the following.
	
	\begin{theorem}[\cite{vanDenDries1998TameTopology}]
		Consider a subset $S \subset \R^m \times \R^n$ which is definable in an o-minimal structure. 
		Then there exists a natural number $N \in \N$, such that for each $a \in \R^m$, if $\abs{S_a}>N$, then $S_a$ is infinite (in other words, $N$ is an upper bound for the cardinality of finite fibers).
	\end{theorem}
	
	We will have need of the following two o-minimal structures.
	
	\begin{definition}
		We say that $f: \R^n \to \R$ is a restricted analytic function if $f\mid_{[-1,1]^n}$ is analytic (meaning $f\mid_{[-1,1]^n}$ can be extended to a real analytic function on an open superset of $[-1,1]^n$) and if $f(x)=0$ for $x \not \in [-1,1]^n$.
		
		Let $\R_{\mathrm{an}}$ be the expansion of $(\R, +, -, \cdot, \le)$ by all restricted analytic functions.
		We say that functions and sets are subanalytic if they are definable in $\R_{\mathrm{an}}$.
		
		Let $\R_{\mathrm{an}, \exp}$ be the expansion of $\R_{\mathrm{an}}$ by the global exponential $\exp: \R \to \R$.
	\end{definition}
	
	\begin{remark}
		The natural logarithm $\log: (0, \infty) \to \R$ is not a subanalytic function in this sense (this follows from the preparation theorems below), but it is definable in $\R_{\mathrm{an}, \exp}$.
	\end{remark}
	
	$\R_{\mathrm{an}}$ and $\R_{\mathrm{an}, \exp}$ are o-minimal, as proven in \cite{Gabrielov1968Projections}, \cite{vanDenDries1986TarskiSeidenberg}, \cite{vanDenDries1988RestrictedElementary} and \cite{vanDenDriesMacintyreMarker1994Elementary}, \cite{LionRolin1997Preparation}, based on \cite{Wilkie96}.
	
	\subsection{Preparation results for subanalytic functions}
	
	For subanalytic functions, there are several preparation results which help us understand their structure.
	To state them, we will first need the notion of \emph{cells}.
	
	\begin{definition}
		A subanalytic cell is a subanalytic set 
		$A \subset \mathbb{R}^{n}$ such that for each $i \in \{1, \dots, n\}$, 
		the set $\pi_{i}(A)$ is either the graph of an analytic subanalytic function on $\pi_{i-1}(A)$, or
		\begin{equation*}
			\pi_{i}(A) = 
			\left\{ y \in \R^i \;\middle|\; y_{<i} \in \pi_{i-1}(A), \ 
			a_i(y_{<i}) \,\Box_1\, y_i \,\Box_2\, b_i(y_{<i}) \right\}
		\end{equation*}
		
		for some analytic subanalytic functions $a_i, b_i : \pi_{i-1}(A) \to \mathbb{R}$ 
		satisfying $a_i(y_{<i}) < b_i(y_{<i})$ on $\pi_{i-1}(A)$, 
		where $\Box_1$ and $\Box_2$ denote either $<$ or no condition.
		
		We say that $A$ is thin if $A$ is the graph of a subanalytic function and that $A$ is fat otherwise.
	\end{definition}
	
	Using cells, we can define the notion of a cell decomposition.
	
	\begin{definition}
		Let $\{C_i\}_{i \in I}$ be a finite collection of cells, such that $C_i \subset \R^n$ for all $i$.
		We say that the $C_i$ form a cell decomposition of $\R^n$ if:
		
		\begin{enumerate}
			\item they form a partition of $\R^n$;
			
			\item inductively, $\{\pi_{n-1}(C_i)\}_{i \in I}$ is a cell decomposition of $\R^{n-1}$.
		\end{enumerate}
		
		We say that a cell decomposition is compatible with a subanalytic set $X$ if $X$ is the union of several of the cells.
	\end{definition}
	
	\begin{theorem}
		Given $X_1, ..., X_k \subset \R^n$ subanalytic, there exists a cell decomposition of $\R^n$ compatible with all $X_i$.
	\end{theorem}
	
	\begin{proof}
		This follows from the preparation theorem below by taking the $f_i$ to be the characteristic functions of the $X_i$.
	\end{proof}
	
	We will informally say that a cell decomposition of $\R^n$ compatible with $X$ is a `cell decomposition of $X \subset \R^n$'.
	
	We will also need the notion of a function `being prepared', for which we need the notions of $\varphi$-functions, special functions and units.
	
	\begin{definition}
		Let $\varphi: X \to \R^m$ and $f: X \to \R$ be functions.
		We say that $f$ is a $\varphi$-function if there exists an analytic function $W: \overline{\varphi(X)} \to \R$ such that $f=W \circ \varphi$.
		Call $f$ a $\varphi$-unit if, moreover, $W \neq 0$ on $\overline{\varphi(X)}$.
		
		Given a subanalytic $f: X \subset \R^{n+1}  \to \R$, we say that $f$ is a special function if there exists a bounded, analytic, subanalytic map $\varphi: X \to \R^{m+2}$ of the form
		
		$$\varphi(x,y)=(a_1(x), a_2(x), ..., a_m(x), a_{m+1}(x) \abs{y}^{-1/p}, a_{m+2}(x) \abs{y}^{1/p}),$$
		where $x \in \R^n$, $y \in \R$ and where $p \in \N \setminus \{0\}$ and $y \neq 0$ on $X$, such that $f$ is a $\varphi$-function.
		
		If $f$ is moreover a $\varphi$-unit, we say that $f$ is a special unit.
		
		We use the same terminology for $f: X \to \C$, where we require that the property holds for $\mathrm{Re}(f)$ and $\mathrm{Im}(f)$.
	\end{definition}
	
	\begin{definition}
		Let $f: X \subset \R^{n+1} \to \R$ be subanalytic.
		We say that $f$ is prepared on $X$ if $X$ is a cell and either $X$ is thin and $f$ is analytic on $X$, or $X$ is fat and
		
		\[f(x,y)=a(x) \abs{y-\theta(x)}^q U(x, y-\theta(x))\]
		
		on $X$.
		Here the following need to hold:
		
		\begin{enumerate}
			\item $a(x)$, $\theta(x)$, $U(x, y-\theta(x))$ are all analytic and subanalytic on $X$.
			\item $q \in \Q$.
			\item $y \neq \theta(x)$ on $X$.
			\item $U$ is a special unit.
		\end{enumerate}
		
		We say that $\theta$ is the center of the preparation.
		
		We say that the preparation is tight if either $\theta=0$ or $y \sim \theta(x)$ on $X$.
	\end{definition}
	
	The following theorem is then shown, for example, by Miller (\cite{M2006}, Main Theorem and Lemma 4.4).
	
	\begin{theorem}[Simultaneous tight preparation] \label{th:prep_subanalytic}
		Let $f_1: X_1 \subset \R^{n+1} \to \R$, $f_2: X_2 \subset \R^{n+1} \to \R$, ..., $f_k: X_k \subset \R^{n+1} \to \R$ be subanalytic.
		Then there exists a cell decomposition $\{C_i\}_{i \in I}$ of $\R^{n+1}$ such that every $f_j$ is prepared on every $C_i \subset X_j$, and moreover this preparation is tight and uses the same center $\theta_i$ for every $f_j$.
	\end{theorem}
	
	Other than simultaneous tight preparation, we will need a more sophisticated preparation result: rectilinear preparation as in \cite{CM2013}, Proposition 5.3.
	
	\begin{definition}
		Let $B \subset \R^n$ be a cell.
		We say that $B$ is $l$-rectilinear if $B=\pi_l(B) \times (0,1)^{n-l}$, such that $\overline{\pi_l(B)}$ is a compact subset of $(0,1]^l$.
		
		We say that $f: B \to \R$ is $l$-prepared if
		
		\[f(x)=x_{>l}^r U(x)\]
		
		for some $r \in \Q^{n-l}$ and where $U$ extends to an analytic subanalytic unit on $\overline{B}$.
	\end{definition}
	
	\begin{theorem}[Rectilinear preparation] \label{thm:rectilinear prep}
		Let $\mathcal{F}$ be a finite set of subanalytic functions on a subanalytic set
		$D \subset \mathbb{R}^{n}$. Then there exists a finite partition $\mathcal{B}$
		of $D$ into subanalytic sets such that for each $B \in \mathcal{B}$ there exist
		$d \in \{0,\ldots,n\}$, $l \in \{0,\ldots,d\}$, and a subanalytic map
		\[
		F : A \to B
		\]
		such that $F$ is an analytic isomorphism, the set
		$A \subset \mathbb{R}^{d}$ is $l$-rectilinear, and each
		function $g$ in the set
		$\mathcal{G}
		=
		\{f \circ F\}_{f \in \mathcal{F}}$
		is $l$-prepared on $A$.
	\end{theorem}
	
	\subsection{Constructible and power-constructible functions}
	
	Finally we come to the definitions of constructible and power-constructible functions.
	
	 \begin{definition}
		Consider a subfield $\K \subset \C$. The algebra of power-constructible functions on a subanalytic set $X$, denoted $\mathcal{C}^{\K}(X)$, is the $\R$-algebra of functions $X \to \C$ generated by all functions of the form: 
		\begin{itemize}
			\item $f$, for $f: X \to \R$ subanalytic.
			\item $\log (g)$ , for $g: X \to \R_{>0}$ subanalytic.
			\item $f^r$, for each  $f\colon X \to \R_{>0}$ subanalytic and $r \in \K.$ 
		\end{itemize}
		When $\K= \Q$, we denote this algebra with $\mathcal{C}(X)$ and we call this the algebra of constructible functions on $X$.
		Note that in this case, we may drop the final item, as $\abs{f}^r$ is subanalytic already when $r \in \Q$ when $f \neq 0$ on $X$.
	\end{definition}
	
	We will now define several further notions that we will need in this paper.
	We start with several further types of cells that will be important.
	
		\begin{definition} \label{def: rectilinear cells}
		Let $B \subset \R^n$ be an $l$-rectilinear cell.
		We denote
		\[B_\varepsilon=\pi_l(B) \times (0, \varepsilon)^{n-l},\]
		and we say that $B_\varepsilon$ is $l, \varepsilon$-rectilinear.
		
		A subanalytic map $f: B_\varepsilon \to \R$ is called $l$-prepared if
		\[f(x)=\prod_{i=l+1}^n x_i^{r_i} U(x),\]
		where all $r_i \in \Q$ and where $U$ extends to an analytic subanalytic unit on $\overline{B_\varepsilon}$.
		
		We call a subanalytic cell $C \subset \R^{n+1}$ $l$-normalized if
		
		\begin{enumerate}
			\item $\pi_n(C)$, which we will call the \emph{base} of $C$, is $l$-rectilinear.
			\item $C$ is of the form
			
			\[C=\{(x,y) \in B \times \R \mid 1<y<a(x)\}\]
			
			for some $l$-prepared $a:B \to \R_{>1}$ or $a(x)=\infty$.
			The map $a$ is called the upper boundary of $C$.
		\end{enumerate} 
		
		Call $C$ \emph{balanced} if $a(x) \sim 1$ on $B$ and \emph{unbalanced} otherwise.
	\end{definition}
	
	\begin{remark}
		Notice that for $a(x)$ as in the above definition, if $a(x)=x_{>l}^Q U(x)$ is the prepared form of $a$, then necessarily $Q \le 0$. Moreover, $C$ is balanced if and only if $Q=0$.
	\end{remark}
	
	\begin{definition}
		Let $C \subset \R^{n+1}$ be an $l$-normalized unbalanced cell.
		Let $B$ be its base and let $a$ be its upper boundary.
		Take any $N > 0$ and let $\varepsilon > 0$ be sufficiently small such that $a(x)>N^2$ on $B_\varepsilon$.
		For such $N$ and $\varepsilon$, define
		\[C_{N, \varepsilon}=\left\{(x,y) \in B_\varepsilon \times \R \mid N<y<\frac{a(x)}{N}\right\}.\]
	\end{definition}
	
	Now, we define what it means for a power-constructible function to be $l$-prepared.
	For this we need the following notion of an $l$-bijection.
	
	\begin{definition}
		Let $F: C \subset \R^{n+1} \to A \subset \R^{m+1}$ be a subanalytic bijection.
		We say that $F$ is an $l$-bijection if:
		
		\begin{enumerate}
			\item $C$ is $l$-normalized.
			\item $F$ is of the form $F(x,y)= (F_1(x), F_2(x,y))$ and $F_1 \colon \pi_n(A) \to \R^m$ is a bijection onto its image.
		\end{enumerate}
	\end{definition}
	
	\begin{definition}
		Let $C \subset \R^{n+1}$ be an $l$-normalized cell and $\K \subset \C$ a subfield.
		If $f: C \to \C$ is $\K$-power-constructible, then we say that $f$ is $l$-prepared on $C$ if it is given as a finite sum of the form
		\[f (x,y)= \sum_b  y^{b} \sum_{rs \beta \gamma} x_{>l}^{r} (\log x_{>l})^{s} y^{\beta} (\log y)^{\gamma} f_{rs\beta \gamma b}(x,y),\]
		with $b\in \K$, $\Re(b) \in [0,1)$, $r \in \K^{n-l}$, $\beta\in \Z$, $s \in \N^{n-l}$, $ \gamma \in \N$. Each $f_{rs\beta \gamma b}: C \to \C$ is a $\psi$-function, where $\psi =(x,x_{>l}^{q_1}y^{-1}, x_{>l}^{q_2}y)$ with $q_1, q_2 \in \Z^{n-l}$ is a bounded analytic subanalytic map.
		If the upper boundary $a$ of $C$ satisfies $a(x)=\infty$, then the $x_{>l}^{q_2} y$-term is not present.
		
	\end{definition}

\section{Preparation results}\label{sec:preparation results}
Throughout this section $\K$ denotes a subfield of $\C$. 
The main result of this section is in Theorem~\ref{th:PrepUnbalanced power constr}, where we prove that every power-constructible function can be written in a particularly useful \enquote{prepared way}, up to a bijection and cell decomposition. 
This preparation result is key for the proof of Theorem~\ref{th:main}. 
It is obtained by combining a variant of rectilinear preparation (cf.~\cite[Prop.\,6.1]{CM2013}) with a careful analysis of power series expansions.

The following proposition combines simultaneous tight preparation with rectilinear preparation.
\begin{proposition}\label{prop:prep for constructible} 
	Let $\mathcal{F}$ be a finite set of constructible functions $f: B \subset \R^{n+1} \to \C$ and $\mathcal{G}$ a finite set of subanalytic functions $g: B \subset \R^{n+1} \to \R$. There exists a partition of $B$ into subanalytic cells $B_i$ such that for each $i$:
	
	\begin{enumerate}
		\item either $B_i$ is the graph of a subanalytic map, or
		\item There is an $l$-bijection $F_i: C_i \to B_i$ and an integer $l \in \{0, \cdots, n\}$ such that:
		\begin{itemize} 
			\item $C_i$ is an $l$-normalized cell.
			\item For each $f \in \mathcal{F},$ we can write $f \circ F_i$ as the following finite sum:
			\[f \circ F_i = \sum_{rs\beta \gamma} x_{>l}^r (\log x_{>l})^s y^\beta (\log y)^\gamma \tilde{f}_{rs\beta\gamma} (x,y),\] 
			where each $s \in \N^{n-l}$, $\gamma \in \N$, $r\in \Q^{n-l}$, $\beta \in \Z$.
			\item For each $g \in \mathcal{G}$, we can write $$g(x,y)=x_{>l}^a y^b \tilde{g}(x,y),$$
			
			with $a \in \Z^{n-l}$ and $b \in \Z$.
			\item There is some function $\psi =(x,x_{>l}^{q_1}y^{-1}, x_{>l}^{q_2}y)$ with $q_1, q_2 \in \Z^{n-l}$, such that all $\tilde{f}_{rs\alpha\beta}$ are $\psi$-functions and all $\tilde{g}$ are $\psi$-units. 
		\end{itemize}
	\end{enumerate}
\end{proposition}
\begin{proof} 
	We will build the bijections $F_i$ by various substitutions.
	It will be obvious at the end that the composition of all of these is an $l$-bijection.
	We may assume that on $B$, each $f \in \cF$ is of the form
	
	\[f(x,y)=\sum_i c_i(x,y) \prod_j \log c_{i,j}(x,y)\]
	
	for subanalytic maps $c_i$ and $c_{i,j}$, where we assume that $c_{i,j}>0$ on $B$.
	We use simultaneous tight preparation (Theorem \ref{th:prep_subanalytic}) on all $c_i, c_{i,j}$ and $g_i \in \mathcal{G}$.
	Fix one of the cells $C_i$ in this decomposition.
	There is a center $\theta(x)$ on $C_i$ and a bounded analytic subanalytic map $\psi(x,y)=(A(x), B(x)y^{-1/p}, C(x) y^{1/p})$ such that, for each $f \in \cF$ and $g \in \cG$, we have
	\[f(x,y)=\sum_i d_{i}(x) \abs{y-\theta(x)}^{\beta_{i}} U_{i}(x,y-\theta(x)) (\log\abs{y-\theta(x)})^{\gamma_i},\]
	\[g(x,y)= d(x) U(x,y-\theta(x)) \abs{y-\theta(x)}^\beta \]
	where the $d_{i}$ are constructible, $d$ subanalytic, the $\beta_i,\beta$ are rational, the $U_{i},U$ are $\psi$-units and the $\gamma_i$ are non-negative integers (all depending on $f$ or $g$).
	After the substitution $y=y'+\theta(x)$, we may assume that $\theta(x)=0$ identically. 
	Then, by splitting $C_i$ into more cells to assume that $y \neq 0$ on $C_i$ and substituting $y=-y'$ if necessary, we may reduce to the case where $y>0$.
	By further splitting $C_i$, we may assume that either $y<1$ or $y>1$ everywhere.
	Finally, after substituting $y = y'^{n}$ for some integer $n$, we may assume that $y>1$ on $C_i$ and that all $\beta_{i},\beta$ are integers and that $p=1$. 
	A final rescaling of $y$ furthermore guarantees that the lower bound of the cell in the $y$-variable can be taken identically equal to 1.
	
	It follows that $C_i$ is of the form
	\[\{(x,y) \in \pi_n(C_i) \times \R \mid 1<y<a(x)\}.\]
	Finally, we use rectilinear preparation as in Theorem \ref{thm:rectilinear prep} on all the subanalytic functions appearing in $d_{ij}(x), d_j(x)$, on the $A(x)$, $B(x)$ and $C(x)$ appearing in $\psi(x)$ and on $a(x)$. The unit $u(x)$ appearing in this step by the rectilinearization of $d_{ij}$ or $d_i$ is again strictly larger than 0 on the closure of the range of $\psi$.
	
	After this, the theorem follows upon noting that the logarithms of $x_{\le l}$ will be subanalytic and any map of the form $x_{\le l}^\gamma$ is automatically bounded away from $0$ and $\infty$, and thus analytic on the closure of $C_i$. This implies that we may take $\psi$ of the form $(x, x_{>l}^{q_1} y^{-1}, x_{>l}^{q_2} y)$. 
	One last rational power substitution on the $x_{>l}$-coordinates allows us to take $a, q_1, q_2, r \in \Z^{n-l}$.
\end{proof}
%
%
The following proposition is only necessary for the study of approximate suprema of \emph{power}-constructible functions which are not constructible.
For constructible functions, Proposition~\ref{prop:prep for constructible} suffices.
\begin{proposition} \label{prop: l-prep power constructible}
	Let $f: B \subset \R^{n+1} \to \C$ be a function in $\mathcal{C}^{\K}(B)$. There exists a partition of $B$ into subanalytic cells $B_i$ such that for each $i$:
	\begin{enumerate}
		\item either $B_i$ is the graph of a subanalytic map, or
		\item there is an $l$-bijection $F_i: C_i \to B_i$ such that 
		\begin{itemize}
			\item $C_i$ is an $l$-normalized cell.
			\item $f$ is $l$-prepared.
		\end{itemize}
	\end{enumerate}
\end{proposition}

\begin{proof}
	By definition, there exist subanalytic functions $f_{hk}$, constructible functions $g_h$ and finitely many $k \in \K$, such that 
	\[f(x,y) = \sum_h g_h(x,y) \prod_{k} \abs{f_{hk}(x,y)}^{k}.\]
	Up to partitioning the domain, we may additionally assume that $f_{hk}(x,y) > 0$ everywhere. 
	Now apply Proposition~\ref{prop:prep for constructible} with $\mathcal{F} = \{g_h\}_h$ and $\mathcal{G} = \{f_{hk}\}_{h,k}$.
	Fix one resulting $l$-normalized cell $C$ and accompanying  $l$-bijection $F \colon C \to F(C) \subset B$.
	On $C$, there exists a bounded function $\psi = (x, x^{q_1}_{>l} y^{-1}, x_{>l}^{q_2}y)$, $q_1,q_2 \in \Z^{n-l}$, such that for each $h$, there are $a_{hk} \in \Z^{n-l}, b_{hk} \in \Z$ and $\psi$-units $\tilde{f}_{hk}$, satisfying:
	\[\prod_k f_{hk}^{k} \circ F(x,y)
		= \prod_k x_{>l}^{a_{hk}k} y^{b_{hk}k}\tilde{f}_{hk}(x,y)^{k} 
		=x_{>l}^{\sum_k a_{hk}k} \cdot \prod_k \tilde{f}_{hk}(x,y)^k \cdot y^{\sum_k b_{hk} k}.\]
	Now, each $\tilde{f}_{hk}$ is a strictly positive $\psi$-unit and thus bounded away from 0. Hence, $\tilde{f}_h = \prod_k \tilde{f}^k_{hk}$ is well-defined and still a $\psi$-function. Moreover, writing $a_h = \sum_k a_{hk} k$ and $b_h = \sum_k b_{hk} k$ we have that
	\[ \prod_{k} \abs{f_{hk} \circ F(x,y)}^k = x^{a_h} \cdot \tilde{f}_h \cdot y^{b_h}. \]
	
	Furthermore, for each $g_h$, we have that
	\[g_h \circ F(x,y) = \sum_{rs \beta \gamma} x_{>l}^{r} (\log x_{>l})^s y^\beta (\log y)^\gamma \tilde{f}_{rs \beta \gamma h}(x,y),\]
	for $r\in \Z^{n-l}$, $\beta \in \Z$, $s \in \N^{n-l}, \gamma \in \N$ non-negative integers and certain $\psi$-functions $\tilde{f}_{r s \beta \gamma h}(x,y)$.
	%
	Combined, this means that
	\[f\circ F(x,y) =  \sum_h  y^{b_h} \sum_{rs\beta\gamma} x_{>l}^{a_h} x_{>l}^{r} (\log x_{>l})^{s} y^{\beta} (\log y)^{\gamma} \tilde{f}_{rs\beta \gamma h}(x,y) \tilde{f}_h(x,y).\]
	Define $f_{rs \beta \gamma h}(x,y)=\tilde{f}_{rs \beta \gamma h}(x,y) \tilde{f}_h(x,y),$ which is still a $\psi$-function. Write
	\[ y^{b_h}= y^{c_h+d_hi} = y^{c_h-\lfloor c_h \rfloor + d_hi} y^{\lfloor c_h \rfloor},\]
	for real numbers $c_h$, $d_h$ and relabel the indices $r,\alpha$ to rewrite $\tilde{f}_{rs \beta \gamma h} x_{>l}^{r+a_h} y^{\beta + \lfloor c_h \rfloor}$ as $\tilde{f}_{rs \beta \gamma h} x_{>l}^{r} y^{\beta}$. Finally, reorder the terms such that $b_h \neq b_{h'}$ if $h \neq h'$, adding the $\psi$-functions $f_{rs\beta \gamma h}$ together if necessary.
	This finishes the proof.  
	\end{proof}
	
	Before proving a preparation result for $l$-prepared functions, we recall the following obvious generalization of Lemma 4.12 from \cite{DVDD1988}. Let $\C\langle X\rangle$ denote the ring of power series over $\C$ in the variables $X_1,..., X_m$ which converge in some neighbourhood of the origin, and let $\C\{X\}$ be the subring of $\C\langle X\rangle$ consisting of the
	power series that converge in a neighbourhood of $[-1,1]^m$. 
	\begin{lemma}\label{veralgemening 4.12 Denef}
		Let $X = (X_1, ..., X_N), Y =(Y_1,\cdots,Y_M), N,M > 0$, and
		$f(X, Y) \in \C\{ X, Y \}$. Then there exists a positive integer $d$ such that $f(X, Y)$ can
		be written as
		$f(X,Y)= \sum_{\abs{i}<d} a_i(X)Y^i u_i(X,Y)$,
		with $a_i(X) \in \C\{X\}$, $u_i(X, Y) \in \C\langle X,Y\rangle$, and $u_i(X, Y)$ a unit in $\C\langle X,Y\rangle$.
	\end{lemma}

\begin{theorem} \label{th:PrepUnbalanced power constr} 
	Let $C$ be an $l$-normalized unbalanced cell with base $B$ and upper boundary $a$.
	Let $f: C \to \R$ be an $l$-prepared fiberwise bounded function in $\mathcal{C}^{\K}$.
	There is a cell decomposition of $\pi_l(B)$ into finitely many cells $\pi_l(B)'$ such that, if we denote by $B'$, $C'$ the cells obtained from $B$ and $C$ by restricting $\pi_l(B)$ to $\pi_l(B)'$, there exist
	
	\begin{enumerate}
		\item numbers $N>0$ and $\varepsilon>0$ such that $C_{N, \varepsilon}'$ is well-defined, which means that $N^2 < a(x)$ on $B'_\epsilon$,
		\item maps $A_{\alpha \beta\gamma}: B_\varepsilon' \to \R$ such that $A_{\alpha 0 \gamma} \in \mathcal{C}^\K(B'_\varepsilon)$,
		\item \label{it:Oz} finitely many continuous maps $f_{\alpha \beta\gamma}: \overline{B_\varepsilon'} \times [0, N^{-1}] \to \R$ with $f_{\alpha \beta\gamma}(x,z) = 1 + O(z)$.
	\end{enumerate}
	
	such that on $C_{N,\varepsilon}'$
	\begin{align*}
		f(x,y) = 	&\sum_{\beta<0} \sum_{\alpha \gamma} A_{\alpha \beta \gamma}(x) f_{\alpha \beta \gamma} (x, y^{-1}) y^{\alpha i + \beta} (\log y)^\gamma \\ 
		+ 	&\sum_{\alpha\gamma} A_{\alpha 0 \gamma}(x)  y^{\alpha i} (\log y)^\gamma \\ 
		+ 	&\sum_{\beta > 0} \sum_{\alpha \gamma} A_{\alpha \beta \gamma}(x) f_{\alpha \beta \gamma}(x, y/a(x)) (y/a(x))^{\alpha i + \beta} (\log (y/a(x)))^\gamma.
	\end{align*}
	We sum over finitely many $\alpha,\beta \in \R$ and $\gamma \in \N$.
	If $a(x)=\infty$, then all terms with $\beta \ge 0$ vanish, except possibly $A_{\alpha 0 0}(x)$ (and we take $y/a(x)=0$).
	
\end{theorem}
\begin{proof}
	We start by assuming $a(x) \neq \infty$.
	By $l$-preparedness of $f$, we can write 
	\begin{equation}\label{eq:prepared}
		f(x,y)= \sum_{b \gamma}(\log y)^{\gamma}  y^{b} \sum_{rs\beta} x_{>l}^{r} (\log x_{>l})^{s} y^{\beta} f_{rs\beta \gamma b}(x,y),
	\end{equation}
	where the $f_{rs \beta \gamma b}$ are $(x, x_{>l}^{q_1} y^{-1}, x_{>l}^{q_2} y)$-functions. 
	Write $\varphi(x,y)=(x, x_{>l}^{q_1} y^{-1},$ $ x_{>l}^{q_2} y)$ and $f_{rs \beta \gamma b} = V_{rs \beta \gamma b} \circ \varphi$.
	
	We claim that, up to partitioning $\pi_l(B)$ (and consequently adapting $B,C$), we may assume there exists fixed $\varepsilon,N$ such that each $V_{rs \beta \gamma b}$ is given by a single convergent power series on an open neighborhood of $\overline{\varphi(C_{N,\varepsilon})}$.
	
	By compactness of $\overline{\pi_l(B)}$, this claim follows if for each $c \in \pi_l(B)$ there exists an open ball $B'$ around $c$ along with $N,\varepsilon$ such that $V_{rs \beta \gamma b}$ is given by a single convergent power series on an open neighborhood of
	\[ \overline{\varphi( (B' \times \R^{n-l + 2}) \cap C_{N,\varepsilon})}. \]
	Note that since $\varphi$ is a bounded function, so are $x_{>l}^{q_1} \cdot 1$ and $x_{> l}^{q_2} a(x)$.
	In particular, $x_{> l}^{q_1}/\sqrt{a(x)}$ and $x_{> l}^{q_2} \sqrt{a(x)}$ tend to zero, whence $(c,0) \in \R^{l + (n -l + 2)}$ belongs to the closure of the range of $\varphi$.
	So let $\varepsilon > 0$ be such that $V_{rs \beta \gamma b}$ is given by a single absolutely convergent power series on $B((c,0),2 \varepsilon) \subset \R^{l + (n- l + 2)}$.
	Without loss of generality, this power series may be taken centered around $(c,0)$.
	In what follows, we shall assume $c = 0$ for convenience of notation. 
	Note that this can be achieved by performing a coordinate change along an $l$-bijection.

	As we already observed, there is some constant $K \in \R_{>0}$ bounding $x_{>l}^{q_1}$ and $x_{> l}^{q_2} a(x)$. 
	This implies that for $(x,y) \in C_{N,\varepsilon}$, both $x_{>l}^{q_1} y^{-1}$ and  $x_{>l}^{q_2} y$ are bounded above by $K/N$.
	Choosing $N$ sufficiently large, it thus follows that 
	\[\varphi\left( (B(c,\varepsilon) \times \R^{n-l + 2}) \cap C_{N,\varepsilon}\right) \subset B( (c,0),\varepsilon). \]
	Hence, taking $B' = B(c,\varepsilon)$ finishes the proof of the claim.

	Next, we group the terms in the sum (\ref{eq:prepared}) above according to like powers of $\log y$ and $y^b$:
	\[f_\gamma^b(x,y)=\sum_{r s \beta} x_{>l}^{r} (\log x_{>l})^{s}  y^{\beta} f_{rs\beta \gamma b}(x,y).\]

	Write $U=x_{>l}^{q_1} y^{-1}$, $V=x_{>l}^{q_2} y$ for the component functions of $\varphi$. 
	We also introduce $T=(\log x_{>l})^{-1}$, $w_{\beta r} = x_{>l}^{r-\beta q_2+B}$, $v_{\beta r} = x_{>l}^{r+\beta q_1+B}$, and

	\begin{align*}
		g_{\gamma}^b(T, x, U, V, w, v) = &\sum_{rs,\beta \ge 0} w_{\beta r}T^{A-s}  V^\beta V_{rs\beta \gamma b}(x, U, V) \\
										&+\sum_{r s, \beta<0} v_{\beta r} T^{A-s} U^{-\beta} V_{rs\beta \gamma b}(x, U, V).
	\end{align*}

	Here $A$ and $B$ are tuples of integers such that all powers of $T$ and the real part of all powers of $x_{>l}$ appearing are strictly positive. 
	Note that
	\begin{align*}
		&(\log x_{>l})^{A} x_{>l}^{-B} f_{\gamma}^b(x,y) \\
		&= g_\gamma^b\left((\log x_{>l})^{-1}, x, x_{>l}^{q_1} y^{-1}, x_{>l}^{q_2} y, (x_{>l}^{r-\beta q_2+B})_{\beta r},(x_{>l}^{r+\beta q_1+B})_{\beta r} \right).
	\end{align*}
	By our claim above, $g_\gamma^b$ is given by a single power series which converges absolutely on a neighborhood of $\C^{l} \times \overline{\varphi(C_{N,\varepsilon})} \times \C^a \times \C^a$, where $a$ is the number of variables of type $w_{\beta r}$ (equivalently $v_{\beta r}$).
	Write $\mathbf{i} = (i_1,\ldots,i_5)$ and let
	\[g_\gamma^b(T, x, U, V, w,v)=\sum_{\mathbf{i}tu} G_{\mathbf{i} tu} T^{i_1} (x_{\le l}^{i_2}) x_{>l}^{i_3} U^{i_4} V^{i_5} w_{t} v_{u}\]
	be this power series expansion.
	Now, for $i\le 0$, we let
	\[g_{\gamma i}^b(T, x, w,v)=\sum_{i_5-i_4=i} G_{\mathbf{i} tu} T^{i_1} (x_{\le l}^{i_2}) x_{>l}^{i_3+i_4 q_1+i_5 q_2}w_t v_u.\]
	Suppose that $a(x)=x_{>l}^Q R(x)$ is the prepared form of $a(x)$.
	For $i>0$, we let
	\[g_{\gamma i}^b(T, x, w,v)=\sum_{i_5-i_4=i} G_{\mathbf{i} t u} T^{i_1} (x_{\le l}^{i_2}) x_{>l}^{i_3+i_4 q_1+i_5 q_2+i Q} w_t v_u.\]
	Note that since $\varphi$ is bounded, all exponents in $x_{>l}$ appearing are non-negative.
	Let $M$ be the largest power of $\log y$ appearing in \eqref{eq:prepared}.
	We now define
	\[g_{\gamma}^{-,b}(T, x, P, w,v)=\sum_{i<0} g_{\gamma i}(T, x, w,v) P^{-i},\]
	\[g_{\gamma}^{0,b} (T, x, w,v)=g_{\gamma 0}(T, x, w,v)\]
	and
	\[g_{\gamma}^{+,b}(T,x, P, w,v, Z)=\sum_{\gamma' \ge \gamma} \binom{\gamma'}{\gamma} Z^{M+\gamma-\gamma'}\sum_{i>0} g_{\gamma i}(T, x, w,v) P^i.\]

	Now from absolute convergence of the power series of $g_{\gamma}^b$ on $\C^{l} \times \overline{\varphi(C_{N, \varepsilon})} \times \C^a \times \C^b$, we can deduce that the power series of $g_{\gamma}^{-,b}$ converges absolutely on $\C^{l} \times \overline{B_\varepsilon} \times [0, N^{-1}] \times \C^a \times \C^b$ (after making $N$ a bit larger and $\varepsilon$ a bit smaller if necessary). Notice as well that there is some complex neighbourhood $D$ of $\overline{B_\varepsilon} \times [0,1/N]$ such that $g_\gamma^{-,b}$ converges on $\C^{l} \times D \times \C^a \times \C^b$.
	We can then use Lemma \ref{veralgemening 4.12 Denef} to write
	
	\[g_\gamma^{-,b}(T, x, P, w,v)=\sum_{i=-k}^{-1} A_{\gamma i}^b(T, x, w,v) P^{-i} f_{\gamma i}^b(T, x, P, w,v).\]
	
	By absorbing constants into the $A_{\gamma i}^b$, we may assume that the units $f_{\gamma i}^b$ are of the form $f_{\gamma i}^b(T, x, P, w,v)=1+P u_{\gamma i}^b(T, x, P, w,v)$, for some power series $u_{\gamma i}^b$.
	This way of writing $g_\gamma^{-,b}$ is valid on a complex neighbourhood $W_2$ of $(0, 0, 0, 0, 0)$.
	Again by compactness of $\overline{\pi_l(B)}$ we may assume that $ \{0\} \times \overline{\pi_l(B)} \subset \pi_n(W_2)$.
	By choosing $\varepsilon$ and $N$ respectively small and large enough, we can also make sure that the expression makes sense on $\overline{B_\varepsilon} \times [0, 1/N]$, after identifying $T=(\log x_{>l})^{-1}$, $w_{\beta s} = x_{> l}^{s- \beta q_2+B}, v_{\beta s} = x_{>l}^{s+ \beta q_1+B}$. Note that $f_{\gamma i}^b$ is continuous and $f_{\gamma i}^b(x,P) = O(P)+1.$
	
	Similarly, $g_\gamma^{0,b}$ automatically already makes sense on $B_\varepsilon$, as $g_\gamma^b$ converges absolutely on $\C^{l} \times \overline{\varphi(C_{N, \varepsilon})} \times \C^a \times \C^a$.
	
	Then the power series for $g_\gamma^{+,b}$ converges absolutely on $\C^{l} \times \overline{B_\varepsilon} \times [0, N^{-1}]\times \C^a \times \C^a \times \C$, after we make $N$ sufficiently large and $\varepsilon$ sufficiently small (which can be seen by letting $P=R(x) y/a(x)$).
	
	Similar to $g_\gamma^{-,b}$ we can then write
	\[g_{\gamma}^{+,b}(T, x, P, w,v)=\sum_{i=1}^k A_{\gamma i}^b(T, x, w,v) P^i f_{\gamma i}^b(T, x, P, w,v)\]
	where $f_{\gamma i}^b(T, x, P, w_1, \cdots, w_m)=1+P u_{\gamma i}^b(T, x, P, w_1, \cdots, w_m)$.
	Let $R$ be an upper bound for $\abs{R(x)}$.
	We can then, similarly to above, assume that this expression makes sense on $\overline{B_\varepsilon} \times [0, (RN)^{-1}]$, after identifying $T=(\log x_{>l})^{-1}$,$w_{\beta s} = x_{> l}^{s- \beta q_2+B}, v_{\beta s} = x_{>l}^{s+ \beta q_1+B}$, $Z=(\log a(x))^{-1}$.
	
	The theorem if $a(x) \neq \infty$ now follows as
	\begin{align*}
		f(x,y&)=\\& (\log x_{>l})^{A} x_{>l}^{-B}\Bigl(\sum_{\gamma b} 	 y^b (\log y)^\gamma g_{\gamma}^{-,b}((\log x_{>l})^{-1}, x, y^{-1},x_{> l}^{s- \beta q_2+B}, x_{>l}^{s+ \beta q_1+B}) \\
		+	&y^b (\log y)^\gamma g_\gamma^{0,b}( (\log x_{>l})^{-1}, x,x_{> l}^{s- \beta q_2+B},x_{>l}^{s+ \beta q_1+B}) \\
		+ 	&y^b (\log y/a(x))^\gamma (\log a(x))^{M} \\
		&  \cdot g_{\gamma}^{+,b}((\log x_{>l})^{-1}, x, R(x) y/a(x),x_{> l}^{s- \beta q_2+B}, x_{>l}^{s+ \beta q_1+B}, \log (a(x))^{-1})\Bigr).
	\end{align*}
	
	Indeed, note that necessarily
	 \[A_{\alpha 0 \gamma}(x)=(\log x_{>l})^A x_{>l}^{-B} g_{\gamma}^{0,b}((\log(x_{>l})^{-1},x, x_{>l}^{s-\beta q_2+B}, x_{>l}^{s+\beta q_1+B}),\]
	  where $b=\alpha i$.
	From the proof it is readily verified that these are $\K$-power-constructible.
	
	Finally, we need to consider what happens if $a(x)=\infty$.
	We can follow a similar approach.
	Note only that we do not have $V$ available.
	Hence, if we have some power $y^{\beta}$ present with $\Re(\beta)>0$, we have to introduce $U^{-\beta}$.
	
	We can then write
	
	\[V_{r s \beta \gamma t}(x, U)=\sum_{i=0}^{-\beta-1} V_{r s \beta \gamma t i} (x) U^i+U^{-\beta} V_{r s \beta \gamma t -\beta}(x, U).\]
	
	Now, expanding $g_{\gamma}((\log x_{>l})^{-1}, x, U)$ like this, we conclude that if it is to be fiberwise bounded, all the coefficients of the negative powers of $U$ must cancel.
	After this, we proceed in exactly the same way.
	
	In the end, it follows that $A_{\alpha 0 \gamma}(x)=0$ for $\gamma \neq 0$ by fiberwise boundedness, which finishes the proof.
\end{proof}

	\section{Unbalanced cells}\label{sec:unbalanced}
	The main goal of this section is to prove Theorem~\ref{th:Unbalanced} below. 
	Essentially, this is a version of the main statement, but only for unbalanced cells, i.e. cells where $\lim_{x\to 0} a(x) = + \infty$ (see Definition~\ref{def: rectilinear cells}). Throughout this section, $\K \subset \C$ denotes a subfield.

	\begin{theorem}\label{th:Unbalanced}
		Let $C \subset \R^{n+1}$ be an $l$-normalized unbalanced cell with base $B$, and let
		$f: C \to \R \in \mathcal{C}^{\K}$ be $l$-prepared and fiberwise bounded on $C$. 
		Then there exist $\varepsilon>0$ and $N>0$ such that $C_{N, \varepsilon}$ is well-defined, and finitely many maps $g_1, g_2, ..., g_k \in \mathcal{C}^{\K}(B_\varepsilon)$ and a constant $C>0$ such that
		
		\[\frac1C \max\{\abs{g_1(x)}, ..., \abs{g_k(x)}\} \le \sup_{y \in C_{N, \varepsilon, x}} \abs{f(x,y)} \le C \max\{\abs{g_1(x)}, ..., \abs{g_k(x)}\}\]
		
		for all $x \in B_\varepsilon$.
		
		Moreover, if $\K \subset \R$, then there exist subanalytic maps $a_1, ..., a_k: B_\varepsilon \to C_{N, \varepsilon}$, for which $\pi_n \circ a_i=\mathrm{Id}$, such that we may take $g_i=f \circ a_i$.
	\end{theorem}
	
	\begin{lemma} \label{functies zijn independent}
		For any $N > 0$, and any finite set  
		\[  A \subset S=\{y^{\alpha} (\log y)^\gamma \mid \alpha \in \C, \gamma \in \N \}\] 
		of functions on a non-empty open interval $I \subset \R_{>0}$, there exist $d_{0},\ldots,d_{\abs{A}-1} \in I$ such that the set of vectors
		\[\{(h(d_0),\ldots,h(d_{\abs{A}-1})) \mid  h \in A \} \subset \C^{|A|} \] 
		is linearly independent. 
	\end{lemma}
	\begin{proof}
		We first note that $S$ is a linearly independent set of functions on $I$.
		Indeed, suppose towards a contradiction that it is not.
		Then there exist $P_0(x),\ldots,$ $P_n(x)$ of the form
		\[ P_j(x) = \sum_{i=0}^{N_j} a_{ij} x^{\alpha_{ij}} ,\]
		with $\alpha_{ij} \in \C$ satisfying $\alpha_{ij} \neq \alpha_{kj}$ for $i \neq k$, and $a_{ij} \in \C \setminus \{0\}$ such that $P_n$ is nonzero and
		\begin{equation} \label{eq:lin_dep}
			f(x) = \sum_{j=0}^n P_j(x) \log(x)^{j} = 0
		\end{equation}
		on $I$.
		
		First suppose that $n > 0$.
		If $P_n$ is non-constant, then take any $\alpha_{i n} \neq 0$ and consider the differential operator $T = (x \frac{d}{dx} - \alpha_{i n})$.
		Then $Tf = 0$ is again a non-trivial equation of the form (\ref{eq:lin_dep}), but with $P_n$ consisting of less terms.
		Inductively, we thus reduce to the case where $P_n$ is constant.
		Then $\frac{df}{dx} = 0$ is still of the form (\ref{eq:lin_dep}), but with smaller $n$.
		
		By induction it thus suffices to deal with the case $n = 0$.
		So take the minimal $N \in \N \setminus \{0\}$ for which there exist $a_0,\ldots,a_N \in \C \setminus\{0\}$ and $\alpha_0,\ldots,\alpha_N \in \C$ such that
		\[ \sum_{i=0}^N a_i x^{\alpha_i} = 0 \]
		on $I$. Without loss of generality $\alpha_0 = 0$. Then after taking the derivative on both sides, we find an equation of the same form, but with less terms.
		This contradiction shows that $S$ must be linearly independent.
		
		Now write $A = \{h_1(x),\dots,h_k(x)\}$ for pairwise distinct $h_i$.
		Note that it follows that $A$ is linearly independent.
		For any choice of $d_0,\dots,d_{k-1} \in I$, it holds that if the set $\{ (h_i(d_0),\ldots,h_i(d_{k-1})):i \in \{1, \cdots, k\}\}$ is linearly dependent, then so is $\{ (h_1(d_i),\ldots,h_k(d_{i})): i \in \{1, \cdots, k-1\}\}$.
		So suppose towards a contradiction that 
		\[\operatorname{span}((h_1(x), \cdots, h_k(x)): x \in I)\]
		has dimension at most $k-1$ and choose a vector $(a_1,\ldots,a_k)$ in its orthogonal complement.
		It follows that for all $x \in I$
		\[ \sum a_i h_i(x) =0,\]
		contradicting the linear independence of $A$.
	\end{proof}

	\begin{lemma}\label{le:uniformUnboundedNeg}  
		Let $K \subset [0, 2 \pi) \times \R_{<0} \times \N$ be a finite set.
		Denote $b(K)=\max\{\gamma \mid (\alpha, \beta, \gamma) \in K\}.$
		Then there exists a real number $\delta(K) > 0$ and constants $d_0, \cdots, d_{\abs{K}-1} \in [1, 2]$ such that the following holds.
		Let $N$ be a fixed positive integer and consider functions $f_{\alpha\beta\gamma}$ such that $\abs{(f_{\alpha\beta \gamma}(y) - 1)(\log y)^{2b(K)}} < \delta(K)$ for all $ y \in (N,+\infty)$ and for $(\alpha,\beta, \gamma) \in K$. 
		Then, for any tuple $c= (c_{\alpha\beta \gamma})_{(\alpha,\beta, \gamma) \in K} \in \mathbb{C}^{\abs{K}}$ the function 
		\[ h(y) = \sum_{(\alpha,\beta, \gamma) \in K}c_{\alpha \beta \gamma} f_{\alpha \beta \gamma}(y) y^{\alpha i}  y^\beta(\log y)^\gamma  \]
		satisfies
		\[\sup_{y>N} \abs{h(y)} \sim \max_{0 \le j <\abs{K}} \abs{h(2Nd_j)},\]
		where the implicit constant depends only on $K$ (and not on the coefficients $c_{\alpha\beta\gamma}$).
		Moreover, there exists rational numbers $C(K), A(K) > 0$ (depending only on $K$) such that for $y > N$:
		\[\abs{h(y)} \le C(K) (2N)^{A(K)} y^{-A(K)} \max_{0 \le j<\abs{K}} \abs{h(2Nd_j)}.\]
	\end{lemma}
	
	\begin{proof}
		As $\sup_{y > N} \abs{h(y)}$ is clearly bounded below by $\max_{0 \le j<\abs{K}} \abs{h(2N d_j)}$ for $d_j \geq 1$, it suffices to prove the `moreover' part.
		Suppose the $f_{\alpha\beta \gamma}$ are given and satisfy $\abs{(f_{\alpha\beta \gamma}(y) - 1)(\log y)^{2b(K)}} < \delta(K)$, for some $\delta(K)$ whose value will be computed during the proof.
		
		Consider the finite-dimensional vector space $V = \C^{\abs{K}}$. For any $c \in V$, define functions on $(N,+\infty)$ and $(0,+\infty)$ respectively:
		\begin{enumerate}
			\item $h_c(y) = \sum_{(\alpha,\beta, \gamma) \in K} c_{\alpha\beta\gamma} y^{\alpha i} y^\beta (\log y)^\gamma  f_{\alpha\beta \gamma}(y)$,
			\item $g_c(y) = \sum_{(\alpha,\beta, \gamma) \in K} c_{\alpha\beta\gamma} y^{\alpha i} y^\beta (\log y)^\gamma $.
		\end{enumerate}
		Denote $k = \abs{K}$. Consider the functions $l_{\alpha\beta\gamma}: \R \to \C:y \mapsto y^{\alpha i} y^\beta (\log y)^\gamma$.
		By Lemma \ref{functies zijn independent}, we can find $k$ points $d_0, \cdots, d_{k-1} \in [1,2]$, such that 
		\[\{(l_{\alpha \beta \gamma}(d_0), \cdots,l_{\alpha \beta\gamma}(d_{k-1})) \mid (\alpha,\beta, \gamma) \in K\}\]
		is a linearly independent set. Then the map 
		\[L \colon \C^{k} \to \C^{k}\colon c \to (g_c(d_0), \cdots, g_c(d_{k-1}))\]
		is an injection, thus $\norm{L(c)}_\infty$ is a norm on $\C^{k}$. 
		Then, since any two norms on the finite-dimensional vector space $V$ are equivalent, there exists some $C(K) > 0$ (depending only on $K$) such that
		\begin{equation} \label{eq:equiv_norms}
			\norm{c}_1 \le C(K) \max_{0 \le j < k} \abs{g_c(d_j)}. 
		\end{equation}
		Denote $K' = \{(\beta, \gamma) \mid \text{there is some } \alpha \text{ with } (\alpha, \beta, \gamma) \in K\}.$ 
		We note that since all occurring $\beta$'s are negative, there exist rational numbers $D(K), A(K) > 0$ (where we need only take $-A(K)>\beta$ for all occurring $\beta$), such that on $(1/2,+\infty)$
		\begin{equation} \label{eq:power-bound}
			\max_{(\beta, \gamma) \in K'}|z^\beta (\log z)^\gamma| \le D(K) z^{-A(K)} .
		\end{equation}
		We may furthermore assume without loss of generality that if $(\alpha,\beta, \gamma+1) \in K$, then also $(\alpha,\beta, \gamma) \in K$.
		It follows that there exists some tuple $c'$ such that $g_c(2Nz)  = g_{c'}(z)$.
		Equations (\ref{eq:equiv_norms}) and (\ref{eq:power-bound}) now combine to show that for $z > 1/2$
		\begin{equation}
			\begin{aligned}
				\abs{g_{c}(2Nz)}	&\le \norm{c'}_1 \max_{(\beta,\gamma) \in K'}|z^\beta (\log z)^\gamma|   \\ 
				&\le \left(C(K)  \max_{0 \le j < k} \abs{g_{c'}(d_j)} \right)  \left(  D(K) z^{-A(K)}\right)  \\
				&=C(K)D(K) z^{-A(K)} \max_{0 \le j<k} \abs{g_c(2N d_j)}.   \label{eq:ineq_g}
			\end{aligned}
		\end{equation}
		
		To continue, we need to consider more closely the relationship between $c$ and $c'$.
		We compute
		
		\begin{equation}
			(2Nz)^{\alpha i}(2Nz)^{\beta} (\log (2Nz))^\gamma=(2N)^{\alpha i + \beta} z^{\alpha i} z^{\beta} \sum_{i=0}^\gamma \binom{\gamma}{i}(\log 2N)^i (\log z)^{\gamma-i}.
		\end{equation}
		This implies that for any fixed $\alpha, \beta$, if we denote $n=n(\alpha, \beta)=\max\{\gamma \mid (\alpha, \beta, \gamma) \in K\}$:
		\[
		\begin{bmatrix}
			c'_{\alpha \beta 0} \\
			c'_{\alpha\beta 1} \\
			c'_{\alpha\beta 2} \\
			\vdots \\
			c'_{\alpha\beta n}
		\end{bmatrix}
		=M_{\alpha \beta}
		\begin{bmatrix}
			c_{\alpha \beta 0} \\
			c_{\alpha\beta 1} \\
			c_{\alpha\beta 2} \\
			\vdots \\
			c_{\alpha\beta n}
		\end{bmatrix},
		\]
		
		with
		\[M_{\alpha \beta}=
		(2N)^{-\alpha i-\beta}
		\begin{bmatrix}
			1 & \binom{1}{1}\log 2N & \binom{2}{2}(\log 2N)^2 & \cdots & \binom{n}{n}(\log 2N)^n \\
			0 & 1 & \binom{2}{1}\log 2N & \cdots & \binom{n}{n-1}(\log 2N)^{n-1} \\
			\vdots & \vdots & \vdots & \ddots & \vdots \\
			0 & 0 & 0 & \cdots & 1
		\end{bmatrix}.\]
		
		Denote by $c_{\alpha\beta}$, $c'_{\alpha\beta}$ the subtuples of $c$, $c'$ consisting of those $c_{\eta\delta \zeta}$ and $c'_{\eta \delta\zeta}$ with $(\eta, \delta)=(\alpha, \beta)$.
		By inverting the matrix in the above equality we can deduce that there is a constant $E(K)$ such that
		
		$$\norm{c_{\alpha \beta}}_1 \le E(K) (2N)^{-\beta} (\log 2N)^{n(\alpha, \beta)} \norm{c'_{\alpha\beta}}_1.$$
		
		Now denote
		
		$$h_c^{\alpha \beta}(z)=\sum_{\{\gamma \mid (\alpha, \beta, \gamma) \in K\}} c_{\alpha \beta\gamma} f_{\alpha\beta\gamma}(z) z^{\alpha i} z^{\beta} (\log z)^\gamma$$
		
		and
		
		$$g_c^{\alpha\beta}(z)=\sum_{\{\gamma \mid (\alpha, \beta, \gamma) \in K\}} c_{\alpha \beta \gamma} z^{\alpha i} z^{\beta} (\log z)^\gamma.$$
		
		We compute
		\begin{equation}
			\begin{aligned}
				| h_c^{\alpha \beta}(2Nz) &- g_c^{\alpha \beta}(2Nz) | 	\\ \le &\max_{(\alpha, \beta, \gamma) \in K} \abs{f_{\alpha \beta \gamma}(2Nz)-1} \norm{c_{\alpha\beta}}_1 \abs{(2Nz)^{\alpha i}} (2Nz)^\beta (\log 2Nz)^{n(\alpha, \beta)}  \\			
				\le & \delta(K) E(K)  z^{\beta} \norm{c'_{\alpha\beta}}_1.
			\end{aligned}
		\end{equation}
		
		We deduce that
		
		\begin{equation}
			\begin{aligned}
				\abs{h_c(2Nz)-g_c(2Nz)} & \le \sum_{\alpha \beta} \abs{h_c^{\alpha \beta}(2Nz)-g_c^{\alpha \beta}(2Nz)}\\
				& \le \delta(K) E(K) z^{-A(K)} \norm{c'}_1\\
				& \le \delta(K) C(K) E(K) z^{-A(K)} \max_{0 \le j <k} \abs{g_c(2N d_j)}.
				\label{eq:ineq_diff}
			\end{aligned}
		\end{equation}
		Hence, taking $\delta(K) =\frac{1}{2 C(K) D(K)}$ (which thus only depends on $K$) we may deduce from the above inequality that
		\[ \abs{g_c(2Nz)} - \frac{1}{2} z^{-A(K)} \max_{0 \leq j < k} \abs{g_c(2Nd_j)} \le \abs{h_c(2Nz)}. \]
		Evaluating at $z = d_j$ for $0\le j < k$ yields
		\begin{equation} \label{eq:ineq_j}
			\frac{1}{2} \max_{0\le j < k} \abs{g_c(2Nd_j)} \le  \max_{0\le j < k} \abs{h_c(2Nd_j)}.
		\end{equation}
		Finally, combining (\ref{eq:ineq_g}), (\ref{eq:ineq_diff}) and (\ref{eq:ineq_j}) shows that for  $z > 1/2$
		\begin{align*}
			\abs{h_c(2Nz)} 	&\le \abs{g_c(2Nz)} + z^{-A(K)} \frac{1}{2} \max_{0\le j < k} \abs{g_c(2Nd_j)} \\
			&\le (C(K) D(K) + \frac{1}{2}) z^{-A(K)} \max_{0 \le j < k} \abs{g_c(2Nd_j)} \\
			&\le 2 (C(K) D(K) + \frac{1}{2}) z^{-A(K)} \max_{0 \le j < k} \abs{h_c(2Nd_j)}.
		\end{align*} 
		A substitution $y=2Nz$ completes the proof.
	\end{proof}
\begin{remark}
If the study of the approximate suprema problem is restricted to constructible functions, Lemmas~\ref{le:uniformUnboundedZero} and~\ref{lemma: no a(x)} can be omitted, since it then suffices to consider Lemma~\ref{le:uniformUnboundedZeroconstr}, which admits a much simpler proof.
\end{remark}
	\begin{lemma}\label{le:uniformUnboundedZero}
		Consider a function
		\[f(y)=\sum_{(\alpha,\gamma)\in K} c_{\alpha \gamma} y^{\alpha i}(\log y)^{\gamma}, \]
		with $1<a<y<b$, $K$ a finite subset of $[0,2 \pi) \times \N $, and $c_{\alpha \gamma} \in \C$.
		Define $\gamma_{\max} = \max\{\gamma\mid (\alpha,\gamma) \in K \text{ for some }\alpha \in \R\}$.
		Let
		\[  p: \C^{\abs{K}} \to \C^{\abs{K}}: c=(c_{\alpha \gamma})_{(\alpha, \gamma) \in K} \mapsto p(c)= \bigl(\sum_{m=\gamma}^{\gamma_{\max}} a^{\alpha i} c_{\alpha m} \binom{m}{\gamma} (\log b/a)^{\gamma} (\log a)^{m-\gamma}\bigr)_{(\alpha, \gamma) \in K}. \]
	 	Then there is some $M_0 \in \R$, such that if $\log\frac{b}{a} > M_0$, the following holds.
		There exist constants $L(K), N(K) > 0$, depending only on $K$, such that
		\[N(K) \norm{p(c)}_\infty \le \sup_{y \in [a,b]} \abs{f(y)} \le L(K)\norm{p(c)}_\infty.\]
		Moreover, there is some $y_0  \in [a^{\frac{3}{4}}b^{\frac{1}{4}},a^{\frac{1}{4}}b^{\frac{3}{4}}]$, which may depend on $c$, and a constant $P(K)$ such that $$\abs{f(y_0)} \geq P(K) \norm{p(c)}_\infty.$$
	\end{lemma}

	\begin{proof}
		We can assume without loss of generality that $K$ is of the form $\{\alpha_1, \cdots, \alpha_m\} \times \{0,1, \cdots, n\}.$ 
		For $c \in \mathbb{C}^{\abs{K}}, M \in \R$, define the function
		\[g_{cM}: [0,1] \to \C: x \mapsto \sum_{(\alpha,\gamma) \in K}  c_{\alpha \gamma} \exp({\alpha Mxi})x^{\gamma}. \] 
		We will prove that there exist constants $M_0, L(K), N(K)$, only depending on $K$, such that for $M > M_0$:
		\begin{equation} \label{eq: intermediate result 1}
			N(K) \norm{c}_2 \le \sup_{[0,1]} \abs{g_{cM}(x)} \leq L(K) \norm{c}_2.\end{equation}
		In order to prove this, we will first show the following. 
		There are $M_0$, $\nu(K), \mu(K)$, such that for each $M > M_0$:
		\begin{equation} \label{eq: intermediate result 1a}
			\nu(K) \norm{c}_2^2 \le \int_{[0,1]} \abs{g_{cM}(x)}^2 dx \le \mu(K) \norm{c}_2^2.
		\end{equation}
		If $\norm{c}_2 = 0$, this is trivial. If not, by dividing each part of Equation \eqref{eq: intermediate result 1a} by $\norm{c}_2^2$, it suffices to prove this for $\norm{c}_2=1.$
		Notice that:
		\begin{equation}\label{eq: equality of 2 integrals}
			\int_{[0,1]} \abs{g_{cM}(x)}^2dx = \int_{[0,1]} \sum_{j=1}^m \left|\sum_{\gamma=0}^n c_{\alpha_j \gamma} x^{\gamma}\right|^2dx + T(c,M),
		\end{equation} 
		with
		\begin{align*}T: (S^1)^{|K|} \times \R \to \R: &(c,M) \mapsto \\ &\int_{[0,1]} \sum_{(\alpha, \gamma) \in K}\sum_{(\alpha',\gamma') \in K, \alpha \neq \alpha'}{c_{\alpha \gamma}} \overline{c_{\alpha'\gamma'}}x^{\gamma+\gamma'}\exp(iMx(\alpha -\alpha'))dx,\end{align*}
		where $S^1$ is the unit circle in $\C$.
		By the Riemann-Lebesgue lemma (see for instance \cite{RiemannLebesgue}), $$\lim_{M \to \infty} T(c,M) = 0.$$ 
		As clearly $T$ is continuous, by compactness of $(S^1)^{\abs{K}}$, for every $\varepsilon > 0$ there is some $M_0(\varepsilon)$, such that for each $M \ge M_0(\varepsilon)$ and $c \in (S^1)^{\abs{K}}$:
		\begin{equation} \label{eq: T(c,M)}
			\abs{T(c,M)} < \varepsilon.\end{equation} 
		Define  $$U: (S^1)^{|K|} \to \R: c \to \int_{[0,1]} \sum_{j=1}^m \abs{\sum_{\gamma=0}^n c_{\alpha_j \gamma} x^{\gamma}}^2 dx = \sum_{j=1}^m\sum_{\gamma, \gamma'=0}^n c_{\alpha_j \gamma} \overline{c_{\alpha_{j} \gamma'}} \frac{1}{\gamma+\gamma'+1}.$$ 
		Since $U$ is a continuous strictly positive function on a compact set, it has a lower bound $\delta$.
		Combining Equation \eqref{eq: equality of 2 integrals} and \eqref{eq: T(c,M)}, it is clear that for each $M > M_0(\delta/2)$:
		\begin{equation}\label{eq: inequality with 1/2} \frac{1}{2} U(c) \le \int_{[0,1]} \abs{g_{cM}(x)}^2dx \le \frac{3}{2} U(c).\end{equation}
		Since $\sqrt{U(c)}$ is a norm on $\C^{|K|}$ and norms on a finite-dimensional vector space are equivalent, we can deduce from Equation \eqref{eq: inequality with 1/2} that there exists constants $\nu(K), \mu(K)$ such that Equation \eqref{eq: intermediate result 1a} holds. 
		This finishes the proof of Equation \eqref{eq: intermediate result 1}, since 
		$$
		\int_{[0,1]} \abs{g_{cM}(x)}^2 dx \le \sup_{[0,1]} \abs{g_{cM}(x)}^2
		$$
		and
		$$\sup_{[0,1]} \abs{g_{cM}(x)}^2 \le \abs{K} \norm{c}_2^2.$$
		Next, we will prove that there is a constant $M_1$ and some function $h: (S^1)^{\abs{K}} \times [M_1,\infty) \to [\frac{1}{4}, \frac{3}{4}]$ and a constant $P(K)$ such that for all $M \geq M_1$ \begin{equation} \label{eq: }
			g_{cM}(h(c,M)) \ge \norm{c}_2 P(K). \end{equation}
		Define  
		\begin{align*}W: (S^1)^{\abs{K}} \to \R: c \mapsto& \int_{[1/4,3/4]} \sum_{j=1}^m \abs{\sum_{\gamma=0}^n c_{\alpha_j \gamma} x^{\gamma}}^2 dx \\ &= \sum_{j=1}^m\sum_{\gamma,\gamma'=0}^n c_{\alpha_j \gamma} \overline{c_{\alpha_j  \gamma'}} \frac{(3/4)^{\gamma+\gamma'+1}-(1/4)^{\gamma+\gamma'+1}}{\gamma+\gamma'+1}.\end{align*}
		 $\sqrt{W(c)}$ is again a norm on $\C^{|K|}$. 
		 So there is some constant $\alpha(K)$ with 
		\begin{equation} \label{eq: W(c)}
			W(c) \ge U(c) \alpha(K).\end{equation} On the other hand, as done previously, we can find some $M_1$, $\beta(K)$, such that for all $M > M_1$, 
			\begin{equation}
				\int_{[1/4,3/4]} \abs{g_{cM}(x)}^2dx \ge \beta(K) W(c).
			\end{equation}
		Thus combining the previous equations, 
		\[ \int_{[1/4,3/4]} \abs{g_{cM}(x)}^2dx \ge P(K) \norm{c}_2^2 \]
		
		for some $P(K)>0$.
	 	Then, using the intermediate value theorem for integrals, there is some $h(c,M) \in [1/4,3/4]$, such $g_{cM}(h(c,M)) \ge 2 \sqrt{P(K)} \norm{c}_2.$
	 	
	 	Notice that \[f\left(a \left(\frac{b}{a}\right)^x\right)=\sum_{(\alpha,\gamma)\in K} c_{\alpha \gamma} a^{\alpha i} \exp(i \log(\frac{b}{a})x)\left(\log a + x \log(\frac{b}{a})\right)^{\gamma}, \]
		Apply Equation \eqref{eq: intermediate result 1} to $x \mapsto f(a (\frac{b}{a})^x)$ and define $y_0 = a \left(\frac{b}{a}\right)^{h(p(c),(\log b/a))}$ corresponding to that function. Finally, one can replace $\norm{\cdot}_2$ by $\norm{\cdot}_\infty$, after adapting the constants, as these norms are equivalent. This finishes the proof.
	\end{proof}

	\begin{lemma} \label{lemma: no a(x)}
		Let $$f: \mathbb{R} \to \mathbb{C}: y \mapsto \sum_{\alpha \in K} c_{\alpha} y^{\alpha i},$$ for some finite set $K \subset [0, 2 \pi)$ and $c_\alpha \in \C$. 
		Fix a positive integer $N$.
		Then there are points $d_0, \cdots, d_{\abs{K}-1} \in [N^2, \infty)$ and $C(K, N)> 0$, both only depending on $K$ and $N$, such that for all $y \in \R$: $$\abs{f(y)} \leq C(K, N) \max_{0 \leq j < \abs{K}}\abs{f(d_j)}.$$
	\end{lemma}
	\begin{proof}
		Denote $$f_c(y) =\sum_{\alpha \in K} c_{\alpha} y^{\alpha i}.$$
		By Lemma \ref{functies zijn independent}, there are $\abs{K}$ points $d_0, \cdots, d_{\abs{K}-1} > N^2$, such that for each $c \in \mathbb{C}^k$, $f_c(d_j) \neq 0$ for some $j \in \{0, \cdots, \abs{K}-1\}$. Thus $\max_{0 \leq j < \abs{K}}\abs{f_c(d_j)}$ is a norm on $\C^{\abs{K}}$.
		As $||c||_1$ is also a norm on $\C^{\abs{K}}$, there exists a constant $C(K)$ such that, for all $c$, $y$: 
		
		\[|f_c(y)| \leq \norm{c}_1 \leq C(K, N) \max_{0 \leq j < \abs{K}}\abs{f_c(d_j)}. \qedhere\]
	\end{proof}
	\begin{lemma}\label{le:uniformUnboundedZeroconstr}
	Consider a function
	\[f: \R \to \R: y \to \sum_{\gamma=0}^d c_\gamma (\log y)^\gamma,\]
	where $1<a<y<b$.
	There exist constants $L(d)$ and $M(d)$, depending only on the degree $d$, such that for all $y \in \R$:
	\[\sup_{y \in (a,b)} \abs{f(y)} \le L(d) \max_{0<  j <M(d)} \abs{f(a^{j/M(d)}b^{1-j/M(d)})}.\]
\end{lemma}

\begin{proof}
	Let $M(d)=d+2$ and let $V$ be the vector space of polynomials of degree $d$.
	Then
	\[\varphi \colon V \to \R^{M(d)} \colon P \mapsto (P(1-j/M(d)))_{0< j<M(d)}\]
	is an isomorphism of vector spaces.
	In particular, $\norm{\varphi(P)}_\infty$ is a norm on $V$.
	Also $P \mapsto \sup_{x \in (0, 1)} \abs{P(x)}$ is a norm on $V$.
	As $V$ is finite-dimensional, these norms are equivalent, so we find that for all degree $d$ polynomials $P$
	\[\sup_{x \in (0,1)} \abs{P(y)} \le C \max_{0<j<M(d)} \abs{P(1-j/M(d))}\]
	for some constant $C$ depending only on $d$.
	Now, from $f$, define $P(x)=f(a(b/a)^x)$.
	The above bound then yields the lemma.
\end{proof}
	\begin{proof}\textit{of Theorem \ref{th:Unbalanced}}
		First assume that $a(x) \neq \infty$. Let $N > 0$ be arbitrary and take $N < y < a(x)$.
		We proceed to find a value for $N$ (and a corresponding value for $\varepsilon$) such that the conclusion of the theorem is satisfied. 
		We use Theorem \ref{th:PrepUnbalanced power constr} to write
		
		\[f(x,y)=f^-(x,y)+f^0(x,y)+f^+(x,y)\]
		
		on $C_{N, \varepsilon}$, where
		\begin{align*}
			f^-(x,y)	&=\sum_{\beta<0} \sum_{\alpha \gamma} A_{\alpha \beta \gamma}(x) f_{\alpha \beta \gamma} (x, y^{-1}) y^{\alpha i + \beta} (\log y)^\gamma, \\
			f^0(x,y)	&=\sum_{\alpha \gamma} A_{\alpha0\gamma}(x)y^{\alpha i} (\log y)^\gamma, \\
			f^+(x,y)	&= \sum_{\beta > 0}\sum_{\alpha \gamma} A_{\alpha \beta \gamma}(x) f_{\alpha \beta \gamma}(x, y/a(x)) (y/a(x))^{\alpha i+\beta} (\log(y/a(x)) )^\gamma.
		\end{align*}
		
		Now let $K^-$ be the set of $(\alpha, \beta, \gamma)$-pairs appearing in $f^-$. We may similarly define $K^+$ as the set of pairs $(\alpha,-\beta,\gamma)$, where $(\alpha, \beta, \gamma)$ appears in $f^+$. We let $\delta(K^+), \delta(K^{-})$ and  $M_0(K_0)$ be as in Lemmas~\ref{le:uniformUnboundedNeg} and \ref{le:uniformUnboundedZero}. Define $b(K^-) = \max\{\gamma \mid (\alpha, \beta, \gamma) \in K^{-}\}$, and similarly define $b(K^+)$.
		Since $f_{\alpha \beta \gamma}(x,y^{-1}) = 1 + O(y^{-1})$ by Theorem~\ref{th:PrepUnbalanced power constr} (\ref{it:Oz}), it follows that $\abs{f_{\alpha \beta \gamma}(x,y^{-1})-1} (\log y)^{2b(K^-)}<\delta(K^-)$ by taking $N$ larger and $\varepsilon$ smaller if necessary.
			Then we may also assume that $\abs{f_{\alpha \beta \gamma}(x,y/a(x))-1}$ $(\log (a(x)/y))^{2b(K^+)}<\delta(K^+)$ everywhere. 
		Let $K^0$ be the set of $(j,l)$-pairs appearing in $f^0$. 
		By taking $\varepsilon$ small enough, we may finally assume that $\log{\frac{a(x)}{N^2}} > M_0(K_0)$.
		
		Now from Lemmas \ref{le:uniformUnboundedNeg} and \ref{le:uniformUnboundedZero}, we can deduce the existence of $d_j^-,d_j^+ \in [1,2]$, and $C(K^-), C(K^+), L(K^0) \in \R$ such that the following estimates hold, for $\varepsilon$ sufficiently small
		
		\begin{align}
			\abs{f^-(x,y)} &\le C(K^-) (2N)^{A(K^-)} y^{-A(K^-)} \max_{0 \le j<\abs{K^-}} \abs{f^-(x, 2Nd_j^-)}, \\
			\abs{f^0(x,y)} &\le L(K^0) \norm{p((A_{\alpha 0 \gamma}(x))_{\alpha \gamma})}_\infty, \\
			\abs{f^+(x,y)} &\le C(K^+) (2N)^{A(K^+)} (a(x)/y)^{-A(K^+)} \max_{0 \le j<\abs{K^+}} \abs{f^+(x, a(x)/ (2N d_j^+))}. 
		\end{align}
		
		Let $C$ be the maximum of $C(K^-)$, $C(K^+)$ and $L(K^0)$, and write
		\begin{enumerate}
			\item $a^-_j(x) =  2Nd_j^-$, 
			\item $a^+_j(x) = a(x) (2Nd_j^+)^{-1}$.
		\end{enumerate}
		Additionally, write $A = \min\{A(K^-),A(K^+)\}$, and consider $P(K^0)$ as in Lemma \ref{le:uniformUnboundedZero}.
		By taking $\varepsilon$ sufficiently small, we may assume that for all $x \in B_\varepsilon$
		\begin{equation*}
			C 2^A a(x)^{-A/4}(2N)^{2A}<\frac{P(K^0)}{3C} \le \frac13.
		\end{equation*} 
		For fixed $x$, let $L$ be the maximal element of the set 
		\begin{align*}
			S=\{& \abs{f^-(x, a_0^-(x))}, \ldots, \abs{f^-(x, a_{\abs{K^-}-1}^-(x))}, \\
			& 3C\norm{p((A_{\alpha 0\gamma}(x))_{\alpha \gamma})}_\infty, \\ 
			& \abs{f^+(x, a_0^+(x))}, \ldots, \abs{f^+(x, a^+_{\abs{K^+}-1}(x))}\}.
		\end{align*}
		
		We will show that $S_f(x)$ is bounded below by $KL$ for some constant $K$, for all $x \in B_\varepsilon$.
		As also $\abs{f(x,y)} \leq DL$ for some large enough constant $D$, it then follows that
		\[S_f(x) \sim \max_j\{ \abs{f(x, a^{-, +}_j(x))}, 3C\norm{p((A_{\alpha 0 \gamma}(x))_{\alpha \gamma})}_\infty\}.\]

		There are three possibilities:
		\begin{enumerate}
			\item $L$ is of the form $\abs{f^-(x, a^-_j(x))}$.
			In this case we have that
			\[\abs{f^+(x, a_j^-(x))} \le C (2N)^A a(x)^{-A} (2Nd_j^-)^A \cdot L<\frac13 L,\]
			since $\varepsilon$ was chosen sufficiently small.
			Moreover,
			\[\abs{f^0(x,y)} \le C  \norm{p((A_{\alpha 0 \gamma}(x))_{\alpha \gamma})}_\infty \le \frac13 L ,\]
			as the $f^0$-terms appear in $S$ with a coefficient of $3C$.
			
			Then we find that
			\[ \abs{f(x, a^-_j(x))} \ge \frac13 L,\]
			which implies that $S_f(x) \geq L/3$, as desired.
			
			\item $L$ is of the form $\abs{f^+(x, a^+_j(x))}$.
			This case finishes in exactly the same way as the previous one.
			
			\item $L$ is of the form $3C\norm{p((A_{\alpha 0\gamma}(x))_{\alpha \gamma})}_\infty$. Taking $y_0$ from Lemma \ref{le:uniformUnboundedZero}, we see that
			$$\abs{f^0(y_0)} \geq P(K_0)\norm{p((A_{\alpha 0 \gamma})_{\alpha \gamma})}_\infty = \frac{P(K_0)}{3C} L,$$
			
			\[\abs{f^-(x, y_0)} \le C (2N)^{A} y_0^{-A}L\]
			
			and
			
			\[\abs{f^+(x, y_0)} \le C \Bigl(\frac{a(x)}{2Ny_0}\Bigr)^{-A} L,\]
			
			where again we see that by taking $\varepsilon$ sufficiently small, we can finish in the same way, (since $y_0 \in [N^{\frac{3}{4}}a(x)^{\frac{1}{4}}, N^{\frac{1}{4}}a(x)^\frac{3}{4}]$), with a lower bound of $\frac{P(K_0)}{9C} L$.
			This concludes the proof in the case $N<y<a(x)$. 
		\end{enumerate}
		Now assume that $N <y$ (so there is no upper bound $a(x)$). Then $f^+ = 0$, and $f^0(x,y) = \sum_\alpha A_{\alpha0}(x) y^{\alpha i}$. 
		So the proof is entirely similar, except that we use Lemma \ref{lemma: no a(x)} to find points $d_1, \cdots, d_{\abs{K_0}-1} > N^2$ with 
		$$f^0(x,y) \leq C(K^0, N) \max_{0 \leq j < \abs{K^{0}}} \abs{f^{0}(x,d_j)}.$$ 
		Finally, if $f \in \mathcal{C}^{\K}$ and $\K \subset \R$, then one can replace $\norm{p((A_{\alpha0\gamma}(x))_{\alpha \gamma})}_\infty$ by $f^0(x,a_0^0(x))$ by Lemma \ref{le:uniformUnboundedZeroconstr}, implying the second statement of Theorem \ref{th:Unbalanced}.\qedhere
	\end{proof}

	\section{Balanced cells}\label{sec:balanced}

	We now turn our attention to balanced cells. The proof in this case is somewhat easier than in the unbalanced case: since $\log(y)$ is subanalytic on balanced cells, the $y$ variable only appears inside subanalytic functions. 
	\begin{theorem}\label{th:Balanced}
		Let $C \subset \R^{n+1}$ be an $l$-normalized balanced cell, and let $f: C \to \C$ be a $\K$-power-constructible $l$-prepared function on $C$, with $\K \subset \C$ a subfield.
		
		Then there exist subanalytic functions $a_1, ..., a_M: \pi_n(C) \to \R$, satisfying $(x, a_i(x)) \in C$ for all $x \in \pi_n(C)$, such that
		
		\[S_f(x) \sim \max_i \abs{f(x, a_i(x))}\]
		
		on the base $B$ of $C$.
	\end{theorem}
	\begin{lemma}\label{le:uniformBounded} 		
		Let $C \subset \C^n$ be a compact set.  
		Consider functions $f_i \colon C\times [0,1]  \to \C$ for $i \in \{1, \cdots, k\}$ satisfying:
		\begin{enumerate}
			\item for each value of $x$, the functions $f_i(x,\cdot)$ are $\C$-linearly independent.
			\item the $f_i$ are analytic on $C \times [0,1]$. 
		\end{enumerate}
		Then there exists some $N \in \N$ such that, as functions of $(c_i)_i \in \C^k$ and $x$,
		\[  \sup_{y \in [0,1]} \abs{\sum_{i=1}^k c_i f_i(x,y)} \sim \max_{0 < j < N} \abs{\sum_{i=1}^k c_i f_i(x,j/N)} .\]
		%
		%
	\end{lemma}
	\begin{proof}
		Write $h(c,x,y) = \sum_{i=1}^k c_i f_i(x,y)$.
		We will show that there exists a constant $M > 0$ such that for each $x$ and all sufficiently large $N$
		\[ \frac1M \norm{c}_2 \le  \sup_{y \in [0,1]}\abs{h(c,x,y)} \le M \norm{c}_2, \]
		\[\frac1M \norm{c}_2 \le  \max_{0<j < N}\abs{h(c,x,j/N)} \le M \norm{c}_2 .\]
		Note that this proves the lemma and that it suffices to show these inequalities in the case $\norm{c}_2 = 1$.
		 
		For the first inequality, let $S^{2k-1} \subset \C^k$ be the unit sphere in $\C^k$ and consider the map
		\[\varphi: S^{2k-1} \times C \subset \C^{k} \times C \to \R: (c, x) \mapsto \sup_{y \in [0,1]} \abs{h(c,x,y)}.\]
		We claim that $\varphi$ is continuous.
		Indeed: this follows from the inequality
		\[\abs{\varphi(c_1, x_1)-\varphi(c_2, x_2)} \le \sup_{y \in [0,1]} \abs{h(c_1, x_1, y)-h(c_2, x_2, y)}\]
		together with uniform continuity of $h$ on $S^{2k-1} \times C \times [0,1]$.
		Hence $\varphi$ has a compact image.
		Moreover, this image cannot contain $0$ by the fiberwise linear independence of the $f_i$.
		We conclude that there exists an $M_1>1$ such that $\mathrm{im}(\varphi) \subset [1/M_1, M_1]$.
		
		To find an appropriate $N$ for the second set of inequalities, note that the real and imaginary parts of $h$ are definable in $\R_\an$.
		Also recall that $h(c,x,\cdot)$ is never identically zero for $(c,x) \in S^{2k-1} \times C$ by the fiberwise linear independence of the $f_i$. 
		By o-minimality of $\R_\an$, it follows that there is some upper bound on the number of zeros, independent of $(c,x)$. 
		Let $N$ be strictly larger than this bound and consider the map
		\[ S^{2k-1} \times C \to \R \colon (c, x) \mapsto \max_{0 < j < N} \abs{h(c, x, j/N)}, \]
		whose image is never zero.
		Now a similar argument as for $\varphi$ establishes the second set of inequalities and finishes the proof.
	\end{proof}
\begin{proof} \textit{of Theorem \ref{th:Balanced}}
	We first observe that on a balanced cell, $\log(y)$ and powers of $y$ are analytic on the closure of the cell.
	Now we use that $f$ is $l$-prepared, and factor out logarithms and powers of $x_{>l}$ to write
	\begin{equation} \label{eq:prepared_balanced}
		f(x,y) =  \log(x_{>l})^S (x_{>l})^R  \sum_{s} (\log x_{>l})^s \sum_r x_{>l}^{r} f_{rs}(x,y),
	\end{equation}
	with $S \in \mathbb{N}^l$, $R \in \Q^l$, all $s \in (\Z_{\le 0})^{l}$, all $\Re(r) > 0$ and where the $f_{rs}$ are $\varphi$-functions for some $\varphi(x,y) = (x,x_{>l}^{q_1}y^{-1}, x_{>l}^{q_2} y)$.
	Clearly, it suffices to approximate the supremum of 
	\[ g(x,y) = \sum_{s} (\log x_{>l})^s \sum_r x_{>l}^{r} f_{rs}(x,y). \]
	By boundedness of $\varphi$, we have that $q_1,q_2 \geq 0$. 
	Since $y$ is moreover bounded away from both $0$ and $\infty$, it follows that $\varphi$ is analytic on the closure of $C$.
	The $f_{rs}$ are the composition of some analytic function on $\overline{\varphi(C)}$ with $\varphi$. 
	Hence, the $f_{rs}$ themselves are analytic on $\overline{C}$.
	Note that if $\K \not\subset \R$, then the $f_{rs}$ are complex-valued in general.
	
	By the same argument, $a(x)$ is analytic on $\overline{C}$, hence we can perform the analytic coordinate change $y=1+u(a(x)-1)$ to write
	\[ f_{rs}(x,y) =  g_{rs}(x,u)  \]
	for functions $g_{rs}$, analytic on $\pi_n(C) \times [0,1]$.
	
	Write $B = \pi_n(C)$ for the base of $C$.
	We will assume in the rest of the argument that $C=B \times [0,1]$.
	By the definition of an $l$-rectilinear cell, $B = \pi_l(B) \times (0,1)^{n-l}$, where the closure of $\pi_l(B)$ is a compact subset of $(0,1]^l$.
	
	Locally around each point $(c,0,d) \in \overline{\pi_l(B)} \times [0,1]^{n-l} \times [0,1]$, there is a neighbourhood where each $g_{rs}(x,u)$ can be written as a single convergent power series.
	Moreover, by compactness of $\overline{\pi_l(B)}$ and $[0,1]$, there is some $\varepsilon > 0$ such that finitely many such neighbourhoods cover $\pi_l(B) \times [0,\varepsilon]^{n-l} \times [0,1]$. 
	Now the complement of $\pi_l(B) \times (0,\varepsilon]^{n-l}$ can be split up in finitely many cells, each of which is $l'$-rectilinear for some $l'>l$, which can vary per cell. 
	Additionally, $f$ is $l'$-prepared on these cells. 
	Indeed, on any rectilinear cell where the coordinate $x_i$ is bounded away from $0$, both $\log(x_i)$ and any powers of $x_i$ are analytic on the closure of that cell.
	
	Hence, by induction on $n-l$, it suffices to approximate the supremum on a neighbourhood of $(c,0,d)$.
	Indeed, in the case $n=l$, we can cover the entire cell by finitely many neighbourhoods of points $(c,d)\in \pi_n(B) \times [0,1]$.
	For notational convenience, we consider the case where both $c$ and $d$ are zero.
	We may thus assume that all $g_{rs}(x,u)$ are given by a single convergent power series centered at $0$.
	
	We now introduce new variables $T_i = (\log x_i)^{-1}$ and $w_r = (x_{> l})^r$ and write
	\[ h(T,w,x,u) =  \sum_{s} T^s \sum_r w_r g_{rs}(x,u) \]
	By Lemma~\ref{veralgemening 4.12 Denef}, we may write (locally around $0$)
	\[ h(T,x,w,u) = \sum_{i < d} b_i(T,w,x) u^i (1 +  u V_i(T,w,x,u)).  \]
	Here the $V_i$ and $b_i$ are analytic in a neighbourhood of the origin.
	Suppose this way of writing $h$ is valid on a polydisk $\{ \abs{z_i} \le \delta \}^M$.
	As the $u^i(1 + uV_i(T,w,x,u))$ are analytic and $\C$-linearly independent as power series for all $T,w,x$, Lemma~\ref{le:uniformBounded} (used with coefficients $c_i=b_i$) yields some $N \in \N$ such that
	\[ \sup_{u \in [0,\delta]} \abs{h(T,w,x,u)} \sim \max_{0 < j < N}\abs{h(T,w,x, j\delta/N)}.  \]
	As for $x_{>l}$ sufficiently small, we have $\abs{T}, \abs{w_r} \le \delta$ (this uses that all $\Re(r) > 0$), it follows that there is some $\varepsilon > 0$ such that for $(x,u) \in [0,\varepsilon]^{m+1}$
	\[ \sup_{u \in [0,\delta]} \abs{g(x,u)} \sim \max_{0 < j  < N}\abs{ g(x,j\delta/N)} . \]
	This is precisely what we still needed to prove.
\end{proof}
\section{Proof of the main theorem}

The results of Sections~\ref{sec:unbalanced} and \ref{sec:balanced} combine to prove our main theorem~\ref{th:main_intro}.
We restate it here for convenience and observe some immediate corollaries.

\begin{theorem}\label{th:main}
	Let $Y \subset \R^{m+n}$ be subanalytic and let $f \colon Y \to \R \colon (x,y) \mapsto f(x,y)$ be a map in $\cC^\mathbb{K}(Y)$ such that $\sup_{y \in Y_x} \abs{f(x,y)}$ is finite for all $x \in \pi_m(Y) \coloneqq X$.
	Then there exist finitely many maps $g_1, g_2, ..., g_k \in \mathcal{C}^{\K}(X)$ and a constant $C>0$ such that
	\[\frac1C \max\{\abs{g_1(x)}, ..., \abs{g_k(x)}\} \le \sup_{y \in Y_x} \abs{f(x,y)} \le C \max\{\abs{g_1(x)}, ..., \abs{g_k(x)}\}\]
	for all $x \in X$.
	Moreover, if $\K \subset \R$, then there exist subanalytic maps $a_1, ..., a_k$ for which $(x,a_i(x)) \in Y$ for all $x \in \pi_m(Y)$ and such that we may take $g_i=f(x,a_i(x))$.
\end{theorem}
%

\begin{proof} 
	We may reduce to the case $n = 1$ by induction.
	Moreover, by Proposition~\ref{prop: l-prep power constructible}, it suffices to show this for an $l$-prepared function $f$ for some $l$.
	Indeed, if $F = (F_1,F_2)$ is an $l$-bijection, then $S_{f \circ F}(x) = S_{f}(F_1(x))$.
	We are done by Theorem \ref{th:Balanced} if $C$ is balanced, so we will assume that $C$ is unbalanced.
	
	By Theorem \ref{th:Unbalanced}, we need to deal only with $C \setminus C_{N, \varepsilon}$.
	Let $B$ be the base of $C$.
	Note that $((B \setminus B_\varepsilon) \times \R) \cap C$ is the union of $l'$-rectilinear cells for several $l'>l$. 
	An induction on $n-l$ finishes this case. Indeed, all $n$-rectilinear cells are necessarily balanced.
	
	So all that remains is dealing with $(B_\varepsilon \times \R) \cap (C \setminus C_{N, \varepsilon})$.
	This set consists of a balanced cell and a set of the form $a(x)/N<y<a(x)$.
	Substituting $y= y' a(x)/N $ also transforms this second part into a balanced cell.
	Since the subanalytic function $a(x)$ is $l$-prepared, the resulting constructible function $\tilde{f}(x,y') = f(x,y' a(x)/N)$ can be (re)written in a prepared form. 
	Note that this rewrite might involve one more power-substitution in $x$ to guarantee that the relevant exponents of $x$ are integral rather than rational.
\end{proof}

If $f$ is everywhere non-negative and the field of exponents is formally real, one function $g(x)$ suffices.

\begin{corollary}\label{cor:fpos}
	Let $Y \subset \R^{m+n}$ be subanalytic, and let $f(x,y) \in \cC^{\mathbb{K}}(Y)$ be such that $\sup_y\abs{f(x,y)} < \infty$ for all $x \in \pi_m(Y)$. 
	Assume additionally that $f(x,y) \geq 0$ for all $(x,y) \in Y$.
	Then, if $\mathbb{K} \subset \R$, there exists $g(x) \in \cC^\mathbb{K}(\pi_m(Y))$ and a constant $C > 0$ such that
	\[ \frac1C g(x) \le \sup_{y \in Y_x} \abs{f(x,y)} \le C g(x).  \]
\end{corollary}
\begin{proof}
	If $f(x,y) \geq 0$ everywhere, then the `moreover' part of Theorem~\ref{th:main} guarantees in particular that $g_i(x) \ge 0$ for all $i$. Since
	\[ g_1(x) + \ldots + g_k(x) \sim \max_{i}\{g_1(x),\ldots, g_k(x)\}, \]
	the corollary follows upon taking $g(x) = g_1(x) + \ldots + g_k(x)$.
\end{proof}

If the $x$-variables live in ambient dimension 1, we obtain explicit bounds on the supremum.
The Adiceam-Cluckers conjecture~\ref{conj cluckers adiceam} follows from Corollary~\ref{cor:CA-almost} after performing a change of variables $x \mapsto 1/x$.

\begin{corollary} \label{cor:CA-almost}
		Let $Y \subset \R^{1+n}$ and  $f(x,y) \in \cC^{\mathbb{K}}(Y)$ be such that $\sup_y\abs{f(x,y)} < \infty$ for all $x \in \pi_m(Y)$.  Assume additionally that $\K \subset \R$.
		Then there exist $r \in \K$, $l \in \N$ and a constant $C > 0$ such that for all sufficiently large $x$
		\[   \frac1C x^{r} (\log x)^{l} \le \sup_{y \in Y_x} \abs{f(x,y)} \le C  x^r (\log x)^l. \ \]
\end{corollary}
\begin{proof}
	Let $g_1, g_2, ..., g_k$ be as in the conclusion of Theorem \ref{th:main}.
	Note that by o-minimality of $\R_{\mathrm{an}, \exp}$, $|g_1|$, ..., $|g_k|$ are also $\K$-power constructible, as the locus $g_i(x)<0$ is a finite union of intervals and points.
	Hence $g(x)=|g_1(x)|+|g_2(x)|+...+|g_k(x)|$ is also constructible and $g(x) \sim \sup_y |f(x,y)|$ on $\R$.
	
	By simultaneous preparation for subanalytic functions (Theorem~\ref{th:prep_subanalytic}), and only looking at the unique resulting cell that is unbounded to the right, we may then write $g$ as a finite sum
	\[ g(x) = \sum_{l} (\log x)^l \sum_{i}  u_{il}(x) x^{r_{il}} .\]
	for $x$ sufficiently large. Here $l \in \N$, $r_{il} \in \K$ and the $u_{il}(x)$ are subanalytic units.
	Let $a_1,\ldots,a_k$ be such that the classes $a_i + \Q \in \R/\Q$ are pairwise distinct and contain the images of all $r_{il}$.
	Then we may write
	\[ g(x) = \sum_{il} (\log x)^l  x^{a_i} \sum_{j=1}^k w_{ijl}(x) x^{r_{ijl}}, \] 
	where all the $w_{ijl}$ are still special units and the $r_{ijl}$ are now all rational. 
	In particular, all the sums $ \sum_{j=1}^N w_{ijl}(x) x^{r_{ijl}}$ are subanalytic and we can apply simultaneous preparation once more to obtain that for all sufficiently large $x$
	\[ g(x) =  \sum_{rl} v_{rl}(x)  x^r (\log x)^l , \]
	for certain special units $v_{lr}$, $r \in \mathbb{K}$ and $l \in \N$. 
	Denote by $(r_0,l_0)$ the (lexicographic) maximum among all $(r,l)$.
	Then $g(x) \sim x^{r_0} (\log x)^{l_0}$, and we are done.
\end{proof}

\section{Applications of the main theorem}\label{sec:corollaries}

\subsection{Measure of flatness}
As explained in the introduction, Corollary \ref{conj cluckers adiceam} above resolves a conjecture of Adiceam and Cluckers stated in \cite{AM2023}. That paper addresses the following question: in \cite{S97}, Sarnak conjectured an upper bound for the number of integer lattice points satisfying a system of polynomial inequalities, see Conjecture \ref{conj: sarnak} above. However, Marmon and Adiceam show that an additional condition related to the measure of flatness is necessary ~\cite[Thm.\,6.1]{AM2023}. For applications of this theorem, an effective way to compute the measure of flatness is useful. Corollary \ref{conj cluckers adiceam} guarantees that this can be done in a finite number of steps, as worked out in \cite{AM2023}. 

Now we can present Theorem 6.1 from \cite{AM2023}, which formulates an answer to Question \ref{question AM} if $\alpha > 1$. Recall that $F = (F_1, \cdots, F_p)$, where $\{F_1(x), \cdots, F_p(x)\}$ is a set of real homogeneous forms of the same degree on $\R^n$, with $p \geq 1$ and $n \geq 2$, and denote $\operatorname{Vol}_n$ for the Lebesgue measure on $\R^n$.
\begin{theorem}[Adiceam-Marmon] \label{thm: solution to sarnaks conjecture for alpha > 1}
	Consider a real number $\alpha > 1$. Assume that $K$ satisfies the assumptions of Conjecture \ref{conj: sarnak}. Moreover, assume that $K$ is semi-algebraic and generic enough so that the singularities of the algebraic variety $\{x \in \R^n: F(x)=0\}$ do not lie on its boundary (see~\cite[p.\,10]{AM2023}). Let $\mathcal{K}$ be a set containing $K$ in its interior, and contained in a set $U$ defined from a support restriction condition (see~\cite[p.\,11]{AM2023}). Denote $\tau_F(\mathcal{K}) = \dim(\{x \in \mathcal{K}\mid F(x) = 0\})$.

	There is some exponent $\delta_F(K, \mathcal{K})$ such that
\[
\begin{aligned}
	\#\Big(
	\{x \in T \cdot K : \norm{F(x)} \le T^{d-\alpha}\}
	\cap \mathbb{Z}^n
	\Big)
	\ll\;&
	\operatorname{Vol}_n\Big(
	\{x \in T \cdot K : \norm{F(x)} \le T^{d-\alpha}\}
	\Big) \\
	&+ T^{\tau_F(\mathcal{K})-\delta_F(K,\mathcal{K})}
\end{aligned}
\]
	holds, if an additional assumption involving the measure of flatness is satisfied.
%
%
%
\end{theorem}
In order to use Theorem \ref{thm: solution to sarnaks conjecture for alpha > 1}, it is necessary to determine the measure of flatness. This is where Conjecture \ref{conj cluckers adiceam} comes into play. 
The measure of flatness quantifies the maximal level of flatness that can arise when intersecting the sublevel set 
\[
\{x \in K : F(x) \le \varepsilon\}
\]
with affine hyperplanes. In particular, the smaller this measure, the flatter the intersection of the sublevel set in some direction. It is defined as follows:
\[q_{F}(\mathcal{K}) = \liminf_{\varepsilon \to 0^+} \frac{\log \mathcal{M}_F(\mathcal{K}, \varepsilon)}{\log \varepsilon},\]
where the function $\mathcal{M}_F(\mathcal{K}, \cdot): \varepsilon > 0 \to \mathcal{M}_F(\mathcal{K}, \varepsilon)$ can be shown to be the parametric supremum of a constructible function. By Corollary \ref{conj cluckers adiceam}, there are constants $c \in \R, \delta \in \R^+_0, a \in \Q, l \in \Z$, such that
\[\delta \cdot c \cdot \abs{\log \varepsilon}^l \cdot \varepsilon^a \leq \mathcal{M}_F(\mathcal{K}, \varepsilon) \leq c \cdot \abs{\log \varepsilon}^l \cdot \varepsilon^a,\]
from which one can deduce, as further explained on p117 of \cite{AM2023}, that
\[\frac{\log \mathcal{M}_F(\mathcal{K}, \varepsilon)}{\log \varepsilon} = q_F(\mathcal{K}) + O\biggl( \frac{\log \abs{\log \varepsilon}}{\log \varepsilon}\biggr).\]
This then shows that $q_F(\mathcal{K})$ is an actual limit instead of a $\liminf$, with explicit error term.

\subsection{Polynomial boundedness of constructible functions} \label{subsec: poly bounded}
Recall that an o-minimal structure is polynomially bounded if every definable univariate function is eventually bounded by a polynomial for sufficiently large $x$. For example, $\mathbb{R}_{\an}$ is polynomially bounded, whereas $\mathbb{R}_{\an, \exp}$ is not. 
\\
In this section, we focus on polynomially bounded classes of functions, rather than o-minimal structures, as even if $\K = \Q$, $\R_{\an, \exp}$, which is the smallest o-minimal structure containing $\mathcal{C}$, is not polynomially bounded. Within this context, we consider a stronger notion of polynomial boundedness: we require that for each multivariate function $f(x)$ in the class, there exists a univariate polynomial $P(\Vert x \Vert)$ that bounds $f(x)$ for all $x$. To see that the multivariate notion is essentially stronger, one can look at the class of continuous log-analytic functions; this class is polynomially bounded in the unary case, but do not remain so in the multivariate setting, as shown in \cite{K23}. 
\\
As an almost immediate consequence of Theorem \ref{th:main}, we obtain that continuous power-constructible functions are polynomially bounded in this stronger sense, as shown in the following corollary. While this statement may seem evident, it is not easily proven directly without the framework developed in this work. Note that the multivariate case does not follow immediately from the unary case. Moreover, continuity is indeed a necessary condition, as demonstrated by the function $f:\mathbb{R}^2 \to \mathbb{R}:(x,y) \mapsto \frac{1}{x-y} \cdot \mathds{1}_{x \neq y}$. Even when restricted to sufficiently large values of $\norm{(x,y)}$, $f$ fails to be polynomially bounded.
\begin{corollary}\label{cor:poly_bdd}
	Continuous power-constructible functions are polynomially bounded: for each continuous $f \in \cC^{\mathbb{K}}(\R^n)$ there exists a univariate polynomial $P(r) \in \R[r]$ such that
	\[ \abs{f(x)} \le P(\norm{x}) \]
	for all $x \in \R^n$.
\end{corollary}
\begin{proof}
	Note that the case $n = 1$ is essentially immediate from the fact that the o-minimal structure $\R_\an$ is polynomially bounded.
	So assume that $n \ge 2$ and define
	\[ g \colon \R \times \R^n \to \R \colon (r,y) \mapsto f(r \cdot y)  .\]
	Then $g$ is also of class $\cC^{\mathbb{K}}$ and satisfies $f(x) = g(\norm{x},x/\norm{x})$ whenever $x \neq 0$. 
	Moreover, by continuity of $f$, $\sup_{\norm{y} =1} \abs{g(r, y)}$ is finite, for all $r \in \R_{>0}$.
	From Theorem~\ref{th:main} we deduce the existence of constructible $g_1,\ldots,g_m  \colon \R \to \R$ such that
	\[ \sup_{||y||=1}\abs{g(r,y)} \le \abs{g_1(r)} + \dots + \abs{g_m(r)}  . \]
	Since the power-constructible univariate functions $g_i$ are continuous at infinity, we may invoke the case $n = 1$ to obtain that each of them is polynomially bounded at infinity. 
	Hence, $f$ is also polynomially bounded at infinity, and hence everywhere, by continuity.
\end{proof}

\subsection{Application to distributions}
In~\cite[Def.\,5.3, Sec.\,8.2]{ACRS24}, the authors introduce $\cC^{\exp}$-class distributions and ask whether they are tempered.
We give partial evidence towards a positive answer: from Corollary~\ref{cor:poly_bdd}, one almost immediately obtains that a certain subclass of the $\cCexp$-class is tempered.

We give a precise definition of this subclass, called the \emph{strict} $\cC$-class below. 
First, recall that the space of continuous functions $\cC^0(\R^n)$ injects into the space of distributions $\mathscr{D}'(\R^n)$.
For any $f \in C^{0}(\R^n)$, we denote by $\partial_i f \in \mathscr{D}'(\R^n)$ its distributional derivative with respect to the $i$-th coordinate.
Given $i \in \N^n$, write $\partial_i \coloneqq \partial_{i_1} \circ \ldots \circ \partial_{i_n}$.

\begin{definition}
	A distribution $u \in \mathscr{D}'(\R^n)$ is of \emph{strict $\cC^{\mathbb{K}}$-class}, if there exists some continuous power-constructible $f \in \cC^{\mathbb{K}}(\R^n)$ along with some $i \in \N^n$ such that
	\[ \partial_{i} f = u. \]
\end{definition} 

\begin{lemma}
	Any distribution of strict $\cC$-class is of $\cC^{\exp}$-class.
\end{lemma}
\begin{proof}
	Let $f \in \cC(\R^n)$ be continuous, let $i \in \N^n$ be arbitrary, and write $u = \partial_i f$.
	We verify the conditions of~\cite[Def.\,4.3]{ACRS24}. We need to check that for a mother wavelet $\Psi$ (\cite[Def.\,4.1]{ACRS24}) it holds that
	\[ (x,\lambda) \mapsto u(\Psi(\lambda (x - \cdot) ))   \]
	is a function of $\cCexp$-class on $\R^n \times \R_{>0}$.
	By definition of distributional derivatives, the above function is of the form
	\[ \int_{y \in \R^n} (-1)^{\abs{i}} (\partial_i \Psi(\lambda (x - y))) f(y) \, dy.  \]
	The claim now follows from the fact that $\cCexp$-functions are stable under differentiation and parametric integration (see~\cite[Thm.\,2.12]{CCMRS18}).
\end{proof}

\begin{corollary}
	All distributions on $\R^n$ of strict $\cC$-class are tempered.
\end{corollary}
\begin{proof}
	As distributional derivatives of tempered distributions are tempered (\cite[Thm.\,7.13]{Rud91}), it suffices to show that each continuous $f \in \cC(\R^n)$ determines a tempered distribution.
	It suffices to show that $\abs{f(x)}$ is bounded above by a polynomial, but this follows from Corollary~\ref{cor:poly_bdd}.
\end{proof}

One can similarly define a notion of distributions of strict $\cCexp$-class, as distributional derivatives of continuous $\cCexp$-functions.
As in the proof above, temperedness of the $\cCexp$-class would follow from combining a polynomial boundedness for $\cCexp$-functions with a proof that all $\cCexp$-class distributions are strict.
Both of these questions are open.


\subsection{A uniform bound for pushforward measures}

As mentioned in~\cite[Rem.\,6.11]{GHS24}, our main Theorem~\ref{th:main} allows for an improvement of Theorem 6.10 in loc. cit., making the exponent in the bound independent of the measures under consideration.
We recall some of the definitions and indicate how the proof changes.

Let $X,Y$ be two smooth $\R$-varieties and let $\varphi \colon X \to Y$ be a jet-flat map. 
Equivalently, $\varphi$ is flat with fibers of semi-log canonical singularities (cf.~\cite[Lem.\,6.5]{GHS24}).
Let $\mu$ be a smooth and compactly supported measure on $X(\R)$. Here, smooth means that for every chart $\psi \colon U \subset X(\R) \to \R^n$, the pushforward measure $\psi_*\mu$ has smooth density with respect to the Lebesgue measure. 
Also take any smooth and non-vanishing measure $\tau$ on $Y(\R)$.

\begin{corollary}\label{corol: pushforward measures}
	There exists a constant $M_\varphi$, only depending on $\varphi$ (and not on $\mu,\tau$) such that for every $p \in \N$ there exists a constant $C_{\varphi,\mu,\tau,p}$ such that for $0<r<1/2$ one has
	\[ \int_{Y(\R)} \left( \frac{(\varphi_*\mu)(B(y,r))}{\tau(B(y,r))} \right)^p \, d\tau(y) < C_{\varphi,\mu,\tau,p} \abs{\log(r)}^{p M_\varphi}. \]
\end{corollary}
\begin{proof}[Proof sketch]
	One reduces to the case where $X$ is affine, $Y = \mathbb{A}^m_\R$ and $\tau$ is the Lebesgue measure on $\R^m$. 
	Now take a non-vanishing smooth measure $\tilde{\mu}$ on $X$ (e.g. induced by a non-vanishing top differential form).
	It suffices to prove the bound for measures of the form $\mu_l =\tilde{\mu}\mid_{B(0,l) \cap X(\R)}$. Now define the constructible function
	\[  G(y,l,r)  = \frac{(\varphi_*\mu_l)(B(y,r))}{r^m},  \]
	and proceed as in the non-archimedean case (\cite[Prop.\,6.8]{GHS24}), now using Theorem~\ref{th:main} instead of the non-archimedean approximate suprema result.
	Note that $G(y,l,r)$ is fiberwise bounded and non-negative, so that Corollary~\ref{cor:fpos} applies.
\end{proof}

\section{Counterexamples}\label{sec:counterexamples}
                                                           
In this section we show that it is not possible to do better than Theorem \ref{th:main_intro} in general.
More precisely, we will show the following theorem.

\begin{theorem}\label{th:counterexamples}
	The following are true:
	
	\begin{enumerate}
		\item\label{it:not_sup_closed} There exists a nonnegative, fiberwise bounded constructible function $f: \R^2 \to \R$ such that $S_f(x)$ is not constructible.
		
		\item\label{it:not_abs} There is no constructible function $f: \R \times \R_{>0} \to \R$ for which there exists a constant $C>0$ such that
		
		$$\frac{\abs{x-\log y}}{C} \le f(x,y) \le C \abs{x-\log y}$$
		
		for all $(x,y) \in \R \times \R_{>0}$.
		
		\item\label{it:not_max} There is no constructible function $f: \R \times \R_{>0} \to \R$ for which there exists a constant $C>0$ such that
		
		$$\frac{\abs{x-\log y}+1}{C} \le \abs{f(x,y)} \le C (\abs{x-\log y}+1)$$
		
		for all $(x,y) \in \R \times \R_{>0}$.
	\end{enumerate}
\end{theorem}

Note that the first point implies that the set of nonnegative, fiberwise bounded constructible functions is not closed under taking parametric suprema.
This implies that we cannot do better than approximating the parametric suprema with constructible functions.

The second point implies that there does not exist a constructible function $f$ such that, for $g: \R^3 \to \R: (x,y,z) \to x-\log y$, $f \sim S_g(x,y)$.
This means that in Theorem \ref{th:main_intro}, we cannot replace $\max\{|g_1(x), ..., |g_k(x)|\}$ by $\max\{h_1(x), h_2(x), ..., h_k(x)\}$ for constructible $h_i$ in general: the absolute values are necessary.

The third point implies that there does not exist a constructible function $f$ such that $|f| \sim S_g(x,y)$ for

$$g: \R^3 \to \R: (x,y,z) \mapsto \begin{cases}
	x-\log y & z\ge 0\\
	 1 & z<0
\end{cases}.$$
In other words, in Theorem \ref{th:main_intro}, we cannot replace $\max\{|g_1(x)|, ..., |g_k(x)|\}$ by $|h(x)|$ for some constructible $h$.

\begin{remark}
	In \cite{K23}, Kaiser provides an example of a constructible function $f(x,y)$ such that $\sup_{y>0} f(x,y)$ is not constructible.
	Notably, that function $f$ is not fiberwise bounded, and hence $\sup_{y>0}\abs{f(x,y)} = + \infty$.
	In particular, that example does not prove item~(\ref{it:not_sup_closed}) of Theorem~\ref{th:counterexamples}. 
\end{remark}

The following lemma will be crucial for the proof of items~(\ref{it:not_abs}) and (\ref{it:not_max}) of Theorem~\ref{th:counterexamples}.

\begin{lemma}\label{le:prepCountEx}
	Let $f: \R^2 \to \R$ be constructible.
	Then there exist a constant $0<c$ and a subanalytic analytic map $a: (c, +\infty) \to \R_{>0}$ such that on the cell
	$$C=\{(x,y) \in (c, +\infty) \times \R \mid y>a(x)\},$$
	we get that
	$$f(x,y)=(x-\log y)^k U_k(x,y)+(x- \log y)^{k-1} U_{k-1}(x,y)+...+U_0(x,y),$$
	where all $U_i$ are constructible.
	Moreover, for each $U_i$ there are a finite set $K_i \subset \Q \times \N \times \Q$ (for which the projection $K_i \to \N$ is injective) and special units $V_{irs\beta}$ such that
	$$U_i(x,y)=\sum_{(r,s, \beta) \in K_i} V_{irs\beta}(x,y) x^r y^\beta (\log x)^s.$$
\end{lemma} 

\begin{proof}
	Suppose that
	$$f(x,y)=\sum_i f_i(x,y) \prod_j \log \abs{f_{ij}(x,y)}.$$
	We use simultaneous tight preparation on all the $f_i$ and $f_{ij}$.
	Next, for every cell $D$ in the resulting cell decomposition, if $f_i(x,y)=a_i(x) \abs{y-\theta(x)}^{q_i} M_i(x,y)$ and $f_{ij}(x,y)=a_{ij}(x) \abs{y-\theta(x)}^{q_{ij}} M_{ij}(x,y)$ are the prepared forms of $f_i$ and $f_{ij}$ on $D$, we use simultaneous tight preparation on all the $a_i$ and $a_{ij}$.
	
	After rewriting everything, on a fixed cell $D$ we get that $f$ is of the form
	$$f(x,y)=\sum_{sj} g_{s j}(x,y) (\log \abs{x-d})^s (\log\abs{y-\theta(x)})^j,$$
	for subanalytic $g_{s j}$, a constant $d$ and a subanalytic analytic map $\theta$.
	Now, for the remainder of the proof, we work on the unique cell $C'$ in this decomposition such that
	$$C'=\{(x, y) \in (c', +\infty) \times \R \mid y>a'(x)\},$$
	for a constant $c'$ and a subanalytic analytic map $a': (c', +\infty) \to \R$.
	By tightness of preparation, on this cell we get $d=\theta=0$, $c'>0$ and that $a'(x)>0$ for $x>c'$.
	Hence
	$$f(x,y)=\sum_{sj} g_{s j}(x,y) (\log x)^s (\log y)^j.$$
	
	Now consider the polynomial $P \in \mathcal{C}(C)[T]$ given by
	$$P(T)=\sum_{sj} g_{sj}(x,y) (\log x)^s T^j.$$
	Using repeated Euclidean division of $P$ by $x-T$ gives us constructible functions $U_i(x,y)$ such that
	$$P(T)=(x-T)^d U_d(x,y)+...+U_0(x,y).$$
	Hence we get
	$$f(x,y)=(x-\log y)^d U_d(x,y)+...+U_0(x,y),$$
	where
	$$U_i(x,y)=\sum_s h_{is}(x,y) (\log x)^s,$$
	for subanalytic $h_{is}$.
	
	Next, we use simultaneous tight preparation on all the $h_{is}$; similar to before, first with respect to the $y$-variable and then with respect to the $x$-variable.
	Let $C$ be the unique cell in this preparation which is of the required form.
	Then by tightness of the preparation, we have written $f$ in the required form.
\end{proof}

\begin{proof}[Proof of Theorem~\ref{th:counterexamples}]

		(\ref{it:not_sup_closed}) Consider the function
		\[ f(x,y) = \begin{cases}
			x \log(y) - y \log(x) + x^2 & \text{if } x > e \text{ and } 1 < y < x, \\
			0							& \text{else}.
		\end{cases} \]
		Note that $f(x,y) \geq 0$ for all $(x,y)$ and that
		\[ \pdv{f}{y} = \frac{x}{y} - \log(x). \]
		Since $\pdv{f}{y} > 0$ for $1 < y < x/\log(x)$ and $\pdv{f}{y} < 0$ for $x/\log(x) < y < x$, it follows that for $x > e$, $f$ attains a unique maximum at $y = x/\log(x)$.
		Hence, for $x > e$
		\[ \sup_y \abs{f(x,y)} = x (\log(x) - \log(\log(x))) - x + x^2. \]
		If the left-hand side were constructible, then $x \log(\log(x))$ would also be constructible on $(e,+\infty)$.
		This is impossible, as any constructible function grows like $x^r \log(x)^l$ at infinity, for some $r \in \Q$ and $l \in \Z$.
	
		(\ref{it:not_abs}) 	Suppose there does exist such a constructible function $f$.
		Using Lemma \ref{le:prepCountEx} (and its notation), we may write
		
		$$f(x,y)=(x-\log y)^k U_k(x,y)+...+U_0(x,y)$$
		
		on a cell $C=\{(x,y) \in (c, +\infty) \times \R \mid y>a(x)\}$.
		By assumption, we get that $f(x, \exp x)=0$ and hence $U_0(x, \exp x)=0$, meaning that
		
		$$\sum_{(r, \beta, s) \in K_0} V_{0rs \beta}(x, \exp x) x^{r} \exp(\beta x) (\log x)^s=0,$$
		
		for all $x$ sufficiently large (where the $V_{0r s \beta}$ are special units) (indeed: note that by subanalyticity of $a$, $(x, \exp x) \in C$ for $x$ large enough).
		However, every term in this sum has a different growth rate; in particular, the only way this sum is zero is if $U_0=0$.
		Hence $f(x,y)=(x-\log y) U(x,y)$ for a continuous function $U$.
		By assumption, it follows that for $x$ large enough and $(x,y) \in C$ with $x \neq \log y$, we get that
		
		$$\frac{|x-\log y|}{C (x-\log y)} \le U(x,y) \le C \frac{|x-\log y|}{x-\log y}.$$
		But then $U$ is a continuous function, changing sign past the curve $x=\log y$ (which is contained in $C$ for $x$ sufficiently large), such that $\abs{U} \ge 1/C$ if $x \neq \log y$, a clear contradiction.
		
		(\ref{it:not_max}) Suppose there does exist such a function $f$.
		Using Lemma \ref{le:prepCountEx}, we may write
		
		$$f(x,y)=(x-\log y)^k U_k(x,y)+...+U_0(x,y)$$
		
		on a cell $C=\{(x,y) \in (c, +\infty) \times \R \mid y>a(x)\}$.
		In particular, $f$ is continuous.
		Since $f \neq 0$ on $C$ by the assumption
		
		$$\frac{1+|x-\log y|}{C} \le |f(x,y)| \le C(1+|x-\log y|),$$ 
		we may assume $f>0$ on all of $C$ (otherwise we can change $f$ by $-f$ without loss of generality).
		
		Note that for any $K>0$, $(x, K \exp x) \in C$ for $x$ sufficiently large.
		Upon substituting $y=K \exp x$ in the given inequality for $f$, we get
		\begin{equation}\label{eq:countExMain}
			\frac{\abs{\log K}+1}{C} \le (- \log K)^k U_k(x, K \exp x)+...+U_0(x, K \exp x) \le C(1+ \abs{\log K}).
		\end{equation}
		Substituting $K=1$ in this equation implies, using notation from Lemma \ref{le:prepCountEx}, that
		$$1/C \le \sum_{(r, s, \beta) \in K_0} V_{0rs\beta}(x, \exp x) x^r \exp(\beta x) (\log x)^s \le C.$$
		
		If this is to be true for arbitrarily large $x$, we need $(r,s, \beta) \le (0,0,0)$ in reverse lexicographic order for all $(r, s, \beta) \in K_0$.
		Moreover, we need $(0,0,0) \in K_0$ for $U_0(x, \exp x)$ to be bounded below by a positive constant.
		Now write $V_{irs\beta}(x,y)=W_{irs\beta}(A_{irs\beta}(x), B_{irs\beta}(x) y^{-1/p_{i r s \beta}})$, where $A_{irs\beta}(x), B_{irs\beta}(x)y^{-1/p_{ i r s \beta}}$ are subanalytic, analytic and bounded on $C$, and $W_{irs\beta}$ is analytic on the closure of the image of $(A_{irs\beta}(x), B_{irs\beta}(x) y^{-1/p_{i r s \beta}})$ (this is possible by the definition of special unit).
		We then see, using the information on $K_0$, that for all $K$,
		
		\begin{equation}\label{eq:counterExU0}
			\lim_{x \to \infty} U_0(x, K \exp x)=W_{0000}(A_{0000}, 0):=q,
		\end{equation}
		
		where $A_{0000}=\lim_{x \to \infty} A_{0000}(x)$.
		
		Next, upon defining
		
		$$U(x,y)=(x- \log y)^{k-1} U_k(x, y)+...+U_1(x, y)$$
		
		and dividing Equation \eqref{eq:countExMain} by $\abs{\log K}$, we find that
		
		\begin{equation}\label{eq:counterEx}
			\frac1C+\frac{1}{C \abs{\log K}} \le \mathrm{sign}(-\log K) U(x, K \exp x)+\frac{U_0(x, K \exp x)}{\abs{\log K}} \le C+\frac{C}{\abs{\log K}}.\end{equation}
		
		Note that from this inequality it follows that $-U(x, \log(\log x) \exp x)$ is bounded above and below by positive constants for large $x$.
		Using the notation from Lemma \ref{le:prepCountEx}, we then get that
		\begin{align*}&-\sum_{i=1}^d \sum_{(r, s, \beta) \in K_i} \\ &(-\log (\log (\log x)))^{i-1} V_{irs\beta}(x, \log(\log x)\exp x)x^r (\log(\log x))^\beta \exp(\beta x) (\log x)^s\end{align*}
		is bounded above and below by positive constants.
		
		Now note that the terms $(\log (\log (\log x)))^{i-1} x^r (\log (\log x))^\beta \exp(\beta x) (\log x)^s$ all have different growth rates.
		If this expression is to be bounded above, then we need $\beta \le 0$ for all $(r, s, \beta) \in K_i$, and moreover if $\beta=0$, we need $r \le 0$.
		If also $r=0$, then necessarily $s=i-1=0$.
		In other words, we need $(i-1, r, s, \beta) \le (0,0,0,0)$ in reverse lexicographic order.
		
		Finally, we see that for $U(x, \log(\log x) \exp x)$ to be bounded below by a positive constant, we need $(0, 0, 0) \in K_1$.
		
		We then get, using the above information on $(i-1, r, s, \beta)$:
		
		\begin{align*}
			\lim_{x \to \infty}& U(x, K \exp x)\\&=\lim_{x \to \infty} \sum_{i=1}^d (-\log K)^{i-1} \sum_{(r, \beta, s) \in K_i} V_{irs\beta}(x, K \exp x) x^r K^\beta \exp(\beta x) (\log x)^s\\
			&= \lim_{x \to \infty} W_{1000}(A_{1000}(x), 0)=W_{1000}(A, 0):=u,
		\end{align*}
		
		where $A=\lim_{x \to \infty} A_{1000}(x)$.
		
		Taking the limit as $x \to \infty$ of \eqref{eq:counterEx} then yields (by the above equality and Equation \eqref{eq:counterExU0}) that
		
		$$\frac1C+\frac{1}{C \abs{\log K}} \le \mathrm{sign}(-\log K) u+\frac{q}{\abs{\log K}} \le C+\frac{C}{\abs{\log K}}$$
		
		for all $K>0$.
		Letting $K \to 0$ yields $u>0$, but letting $K \to \infty$ yields $u<0$, a contradiction.
\end{proof}

	\bibliographystyle{amsplain}
	\bibliography{ref}

\end{document}